\documentclass[a4paper,10pt,twoside]{article}

\usepackage[margin=1.5in, marginparwidth=1.2in]{geometry}
\usepackage[numbers]{natbib}
\usepackage[utf8]{inputenc}

\usepackage{amsmath} 
\usepackage{mathrsfs}
\usepackage{amsthm} 
\usepackage{amsfonts}
\usepackage{amssymb, latexsym}
\usepackage{mhequ}
\usepackage{verbatim}

\usepackage{kerkis}

\usepackage{graphicx}
\usepackage{color}
\usepackage{xcolor}
\usepackage{tikz}
\usetikzlibrary{calc,intersections,through,backgrounds}
\usetikzlibrary{decorations.pathreplacing}
\usetikzlibrary{shapes}

\usepackage{float} 
\usepackage{enumerate}

\usepackage[pdfauthor={}, pdftitle={}, pdftex]{hyperref}
\hypersetup{
	colorlinks=false,
	pdfborder={0 0 0}
}

\usepackage{tocloft}
\usepackage{fancyhdr}

\usepackage[]{appendix}

\usepackage{hyperref}
\hypersetup{
	colorlinks=true,    
	linkcolor=blue,     
	citecolor=red,    
	filecolor=blue,   
	urlcolor=cyan,
}

\allowdisplaybreaks

\theoremstyle{plain}
\newtheorem{thm}{Theorem}[section]

\newtheorem{lem}[thm]{Lemma}

\newtheorem{ass}[thm]{Assumption}

\numberwithin{equation}{section}
\theoremstyle{remark}
\newtheorem{remark}[thm]{Remark}

\theoremstyle{definition}
\newtheorem{definition}[thm]{Definition}
\newtheorem{example}[thm]{Example}

\newcommand{\N}{\mathbb{N}}
\newcommand{\R}{\mathbb{R}}
\renewcommand{\min}{\operatorname{min}}
\renewcommand{\max}{\operatorname{max}}
\newcommand{\diag}{\operatorname{diag}}

\renewcommand{\div}{\text{div}}
\newcommand{\defeq}{\mathrel{:=}}
\renewcommand{\d}{\text{d}}
\newcommand{\norm}[1]{\left\lVert #1 \right\rVert} 
\newcommand{\zorm}[3]{\left\lVert #1 \right\rVert_{L_x^{#2}L_v^{#3}}} 		
\newcommand{\dorm}[4]{\left\lVert #1 \right\rVert_{L_t^{#2}L_x^{#3}L_v^{#4}}} 
\DeclareMathOperator{\tr}{tr}

\title
{Existence of martingale solutions to a stochastic kinetic model of chemotaxis}
\author{Benjamin Gess, Sebastian Herr, Anne Niesdroy}
\date{}

\fancypagestyle{plain}{

	\fancyhf{}
	\fancyfoot[L]{\footnotesize \today}
	\fancyfoot[C]{\thepage}
}

\renewenvironment{abstract}
{
	\begin{minipage}{.9\textwidth}\small\textbf{Abstract}.\noindent
	}
	{
	\end{minipage}
}

\newenvironment{keywords}
{
	\begin{minipage}{.9\textwidth}\small\textbf{Keywords}:\noindent
	}
	{
	\end{minipage}
}

\newenvironment{msc}
{
	\begin{minipage}{.9\textwidth}\small\textbf{MSC 2010}:\noindent
	}
	{
	\end{minipage}
}

\begin{document}
	\maketitle	

\begin{center}
	\begin{abstract}
	We show the existence of local and global in time weak martingale solutions for a stochastic version of the	Othmer-Dunbar-Alt kinetic model of chemotaxis under suitable assumptions on the turning kernel and stochastic drift coefficients, using dispersion and stochastic Strichartz estimates. The analysis is based on new Strichartz estimates for stochastic kinetic transport. The derivation of these estimates involves a local in time dispersion analysis using properties of stochastic flows, and a time-splitting argument to extend the local in time results to arbitrary time intervals. 
	\end{abstract}	
\end{center}

\bigskip

\begin{keywords}
Stochastic dispersion and Strichartz estimates, martingale solutions, stochastic kinetic chemotaxis.
\end{keywords}

\smallskip

\begin{msc} 35R60 \end{msc}

\tableofcontents
\section{Introduction}
In this paper, we examine a stochastic extension of the Othmer-Dunbar-Alt kinetic model of chemotaxis. Chemotaxis refers to the movement of bacteria, such as E. Coli, in response to chemical gradients. This movement is characterized by alternating phases of random, directed movements ('runs') and abrupt changes in direction ('tumbles'). 
To describe the stochastic 'run', let $\{\beta^k\}_{k \in \N}$ be one-dimensional Brownian motions and $\sigma^k: \R^{2d} \rightarrow \R^d$ be divergence free vector fields. Physically, we view the quantity
\begin{align*}
	(t,x,v) \mapsto \sum_{k \in \N} \sigma^k(x,v) \circ \d \beta^k_t
\end{align*}
as a random environment in which the bacteria move, and which is colored in space and white in time.\footnote{In order to simplify notation we may frequently omit the $\omega$-dependency.} In the absence 
of chemoattractant substances all particles evolve according to the SDE
\begin{align}\label{SDE: stoch. flow_intro}
	\d X_t = V_t \d t, \quad \d V_t = \sum_{k \in \N} \sigma^k(X_t, V_t) \circ \d \beta^k_t
\end{align}
and are only distinguished from one another according to their initial location. The superposition of their stochastic 'run' with the 'tumbling' described by a turning kernel $K(S)$ leads to a stochastic kinetic equation for the distribution of cells in phase space $f(\omega, t, x, v)$ given by 
\begin{align*}
		\d f + v \nabla_x f \d t + \sum_{k} \div_v(f\sigma^k \circ \d \beta^k) &= \int_{V} K(S)f' - K^{\ast}(S)f \d v' \d t,
\end{align*}
with compact $V \subseteq \R^d$, turning kernel $K(S)$ supported in $\R \times \R^d\times V\times V$ and 
\begin{align*}
	\int_{V} K(S) f' \d v' &= \int_V K(S)(t,x,v,v') f(\omega,t, x, v') \d v'\\
	\int_{V} K^{\ast}(S) f \d v' &= \int_V K(S)(t,x,v',v) f(\omega, t, x, v) \d v'.
\end{align*}
We analyze the system given by this equation, coupled with an elliptic equation governing the concentration of the chemical attractant $S(\omega,t,x)$ given by \begin{align*}
		S - \Delta S &= \rho = \int_{\R^d} f(\omega, t, x, v) \d v.
\end{align*}
Thus, combining these equations we are interested in the local and global in time existence of solutions for the following chemotactic model in the presence of an external random force
\begin{align}
\begin{split}\label{IVP_stoch}
			\d f + v \nabla_x f \d t + \sum_{k} \div_v(f\sigma^k \circ \d \beta^k) &= \int_{V} K(S)f' - K^{\ast}(S)f \d v' \d t\\
	f(\omega, 0, x, v) &= f_0(x, v)	\\
S - \Delta S &= \rho = \int_{\R^d} f(\omega, t, x, v) \d v. 
\end{split}
\end{align}
The deterministic \textit{kinetic} model of chemotaxis derived by Othmer, Dunbar and Alt \cite{ADO1988} can be thought of as a mesoscopic analogue of the Keller-Segel model which was introduced by Keller and Segel in the 1970s (\cite{KS1970}, \cite{KS1971}). Othmer, Dunbar and Alt derive this model from a correlated random walk.
When deriving their model, they include an external forcing acting on the individuals (cf. \cite[Equation (33)]{ADO1988}). Since this forcing corresponds to the stochastic term in equation \eqref{IVP_stoch} we will refer to it as the external random force.

In this work, we prove the existence of weak martingale solutions under the following regularity assumption on the turning kernel $K$ and suitable assumptions on the stochastic drift coefficients $\sigma^k$. 

\begin{ass}\label{ass:turning_kernel}
	Let $K: L_t^{r_1}W_x^{1,p_1} \rightarrow L_t^{r_1}L_x^{p_1}L_v^{p_2}L_x^{p_3}$. For all $p_2,p_3 \in [1, \infty]$ and all bounded $V \subseteq \R^d$ there exists 	$C(|V|, p_2, p_3)$ such that for all $p_1 \geq \max(p_2,p_3)$ and all $r_1 \in [1,\infty]$ and all $S \in L_t^{r_1}W_x^{1,p_1}$ and all $t\in (0,\infty)$ we have that $K(S) \geq 0$ and satisfies
	\begin{align*}
		\left\Vert K(S) (t,x,v,v') \right\Vert_{L_x^{p_1} L_v^{p_2} L_{v'}^{p_3}}
		\le 		C(|V|, p_2, p_3) 
		\cdot 	\left(
		\Vert S(t, \cdot ) \Vert_{L^{p_1}} 
		+ \Vert \nabla S(t, \cdot ) \Vert_{L^{p_1}}
		\right).
	\end{align*}
\end{ass}

\begin{remark}
The existence of weak solutions under Assumption \ref{ass:turning_kernel} on the turning kernel $K$ in the \textit{deterministic} setting  was shown by Bournaveas, Calvez, Guit\'{e}rrez and Perthame in \cite[Theorem 3]{BCGP2008}. This assumption is weaker than assumptions used before since there is no assumption of delocalization. Existence of weak solutions under a delocalised assumption was for example shown in \cite[Theorem 1]{CMPS2004}.\\
	Delocalisation is a restrictive assumption that is not satisfied in many relevant biological settings, where turning kernels typically rely only on the turning angle $\theta(v,v') = \arccos\left(\frac{vv'}{\norm{v}\norm{v'}}\right)$, sometimes combined with a sum of nascent delta functions. For example, turning kernels of this form are derived in \cite[p.278]{ADO1988} and \cite[p. 856]{EO2007}. Precisely, this corresponds to kernels of the form
	\begin{align*}
		K_{\varepsilon}(S)(t,x,v,v') = \lambda(S)(x,t) h(\theta (v,v'))\delta_{\varepsilon}(\norm{v}-\norm{v'}), 
	\end{align*}
	where the rate $\lambda(S)$ and the function $h$ satisfy appropriate regularity conditions. 
	If $\lambda(S)$  satisfies Assumption \ref{ass:turning_kernel} but not a delocalised assumption and $h$ is bounded, then the turning kernel $K$ satisfies Assumption \ref{ass:turning_kernel} but not a delocalised assumption.
\end{remark}

Analyzing \eqref{IVP_stoch} presents three main challenges. First, proving existence results requires a-priori estimates, but, as in the deterministic setting, the regularity of the kernel is insufficient to compute the $L^1$-norm of the integral terms on the right-hand side of equation \eqref{IVP_stoch}. In the deterministic setting, Bournaveas, Calvez, Guit\'{e}rrez and Perthame \cite[Theorem 3]{BCGP2008} address this issue by deriving Strichartz estimates in mixed $L^p$-spaces. Accordingly, in the present work we overcome this difficulty by establishing novel pathwise Strichartz estimates for stochastic kinetic transport. To our knowledge, this is the first analysis of a stochastic version of the kinetic chemotaxis model using Strichartz estimates. In the context of the stochastic Boltzmann equation, Punshon-Smith and Smith \cite{PS2018} address related regularity issues. They resolve this difficulty by making use of renormalized martingale solutions. In contrast, relying on Strichartz estimates in mixed $L^p$-spaced allows us to analyze the system \eqref{IVP_stoch} directly without renormalization. 

Second, when deriving stochastic Strichartz estimates, we have to control the dispersion. Since, in contrast to the deterministic case, stochastic characteristics cannot be calculated explicitly, we have to carefully analyze properties of the stochastic flow in order to show dispersion. With the help of various results by Kunita \cite[Theorem 1.4.1, Theorem 3.4.1, Theorem 4.3.2., Theorem 4.6.5, Example on pages 106f.]{K1997} we establish local in time dispersion under regularity assumptions on the stochastic drift coefficients $\sigma^k$. Subsequently, we apply a time-splitting argument inspired by \cite{GS2015}, conservation of mass and duality arguments, to show Strichartz estimates that are valid on arbitrary time intervals. 

Third, previous works on stochastic (nonlinear) transport \cite{LPS2013b, LPS2014, GS2015} developed a purely path-by-path approach to the SPDEs, based on a transformation argument, without giving meaning to the stochastic integrals or the SPDE itself. Since we want to give a stochastic meaning to the resulting process, we ultimately integrate this pathwise technique with the concept of martingale solutions. In comparison to the deterministic setting, we have to treat the adaptedness of the solution as each weak solution is associated with its own stochastic basis.

The key assumptions on the stochastic drift are expressed in terms of the stochastic flow $\Phi_{s,t}(x,v)$ associated with the SDE \eqref{SDE: stoch. flow_intro}. To ensure that this stochastic flow exists globally, is unique, and is volume-preserving the following assumption is sufficient. 

\begin{ass}\label{ass: ex_flow}
Consider the SDE \eqref{SDE: stoch. flow_intro}. Assume that $\sigma^k \in C^1(\R^{2d})$ and that $\div_v \sigma^k = 0$ for all $k$. Furthermore, assume that the local characteristic corresponding to \eqref{SDE: stoch. flow_intro} belongs to the class $B_b^{0,1}$ (see Definition \ref{def: local_characteristic_semi} and Definition \ref{def: local_characteristic_class_B_b} for a precise definition of the local characteristic and this class).
\end{ass}

In this work, we prove that local or global in time dispersion and Strichartz estimates are valid for the stochastic kinetic transport equation if one of the following assumptions on the stochastic flow is satisfied. 

\begin{ass}\label{ass:disp_local}
	Fix $T \in (0,\infty)$. Assume that Assumption \ref{ass: ex_flow} is fulfilled and that there exists a constant $C$ and a $\mathbb{P}$-a.s. positive stopping-time $\tau(\omega)$ with $0\leq \tau(\omega) \leq T$ such that
	\begin{align*}
		\left|\det D_v \Phi_{s,t}(x,v)^{(1)}\right|& \geq C |t-s|^d \quad &\text{for all } |t-s| \leq \tau (\omega).
	\end{align*}
\end{ass}

 In Section \ref{subsec:dispersion_local} we show that further regularity and boundedness assumptions on $\sigma^k$ imply that Assumption \ref{ass:disp_local} is always fulfilled. If we can further ensure the following more restrictive assumption that we have local in time dispersion up to a fixed deterministic constant $\tau$, we are able to show global in time existence of a weak martingale solution to \eqref{IVP_stoch}. 

\begin{ass}\label{ass:disp_glob}
	Fix $T \in (0,\infty)$. Assume that Assumption \ref{ass: ex_flow} is fulfilled and that there exist constants $C$ and $\tau$ independent of $\omega$ with $0<\tau\leq T$, such that for all $\omega \in \Omega$ 
	\begin{align*}
		\left|\det D_v \Phi_{s,t}(x,v)^{(1)}\right|& \geq C |t-s|^d
		\quad &\text{for all } |t-s| \leq \tau.
	\end{align*}
\end{ass}
We verify Assumption \ref{ass:disp_glob} for some classes of coefficients $\sigma^k$ in Section \ref{subsec: dispersion_glob}. 

The following main theorem states the existence of a weak martingale solution to \eqref{IVP_stoch}. 

\begin{thm}\label{thm:main}
	Let $d \ge 2$. Fix $T \in (0, \infty)$ and consider parameters $r,a,p,q$ such that 
	\begin{align*}
		r \in \left(2, \frac{d+3}{2}\right], 
		\quad
		r \ge a \ge 
		\max\left(\frac{d}{2},\frac{d}{d-1}\right) 
		\quad
		\frac{1}{p} = \frac{1}{a} - \frac{1}{rd}, \frac{1}{q} = \frac{1}{a} + \frac{1}{rd}.
	\end{align*} 
	Consider \eqref{IVP_stoch} nonnegative initial data $f_0:\R^d \times \R^d \rightarrow \R$ that is supported in $\R^d \times V$ such that $f_0 \in L^1(\R^{2d}) \cap L^{a}(\R^{2d})$ and $\Vert f_0\Vert_{L^a(\R^{2d})}$ 
	is sufficiently small. Assume that for all $r_1,p_1 \in [1,\infty]$ and $S \in L_t^{r_1}W_x^{1,p_1}$ the turning kernel $K(S)$ is supported in $\R \times \R^d \times V \times V$ and fulfills Assumption \ref{ass:turning_kernel}. 
	\begin{enumerate}
		\item Assume that $\sigma^k$ satisfy Assumption \ref{ass:disp_local} for some stopping-time $\tau$. Then, there exists a $\mathbb{P}$-almost surely positive stopping-time $\tilde{\tau}$ depending on $\norm{f_0}_{L^{a}(\R^{2d})}$ and $\tau$ such that \eqref{IVP_stoch} has a nonnegative weak martingale solution on $[0,\tilde{\tau}]$. 
		\item Assume that $\sigma^k$ satisfy Assumption \ref{ass:disp_glob} for some deterministic constant $\tau$. Then, \eqref{IVP_stoch} has a nonnegative global in time weak martingale solution on $[0,T]$.
	\end{enumerate}
\end{thm}

\subsection{Comments on the literature}
The existence of solutions for the deterministic kinetic model of chemotaxis without external forcing was, for example, studied in Hillen and Othmer (\cite{HO2000}, \cite{HO2002}), Chalup et. al. {\cite{CMPS2004}}, Hwang et al. (\cite{HKS2005a}, \cite{HKS2005b}, \cite{HKS2006} and Perthame {\cite{P2004b}}. They use different assumptions on the boundedness of the turning kernel and initial value. A key ingredient in their assumptions is its delocalized structure. 
In 2008, Bournaveas, Calvez, Guit\'{e}rrez and Perthame first used dispersion and Strichartz estimates in order to show the global in time existence of solutions of the deterministic analogue of \eqref{IVP_stoch} (\cite[Theorem 3]{BCGP2008}) in dimension $d=3$ and with integrability in time $r =3$. With this argument, delocalization is no longer needed. Under some regularity and boundedness conditions on $\sigma^k$ we extend their result to the stochastic case. In addition, in combination with the $TT^{\ast}$ argument we extend it to any dimension $d \ge 3$ and a wider class of parameters $r, a, p$ and $q$. \\
 A comprehensive analysis of deterministic kinetic transport equations which includes most of the deterministic analogues of the statements in Section \ref{sec: strichartz_stoch} can be found in \cite[Section 2]{P2004}. Dispersion and Strichartz estimates were discussed in \cite[Théorème 2 and Théorème 1]{CP1996} for the first time. Later, Ovcharov improved the inhomogeneous Strichartz estimate using the $TT^{\ast}$-argument, as detailed in \cite[Therorem 2.4.]{O2011}.

There are other works and techniques dealing with stochastic kinetic transport equations. In \cite{PS2018} Punshon-Smith and Smith show the existence of renormalized martingale solutions for the Boltzmann equation. Their work is a stochastic extension of a work by Di Perna and Lions \cite{DL1989b}. They use the concept of renormalized solutions whereas we make use of mixed $L^p$-spaces. In \cite{BP2022} Bedrossian and Papathanasiou use energy-methods to show the local well-posedness for Vlasov-Poisson and Vlasov-Poisson-Fokker-Planck systems in
stochastic electromagnetic fields. In comparison to their work, we rely on Strichartz estimates, that enable us to work in mixed $L_t^rL_x^pL_v^q$-spaces for different choices of $r$, $p$ and $q$, and combine the pathwise analysis with the concept of martingale solutions. 

Properties of stochastic flows have been discussed extensively by Kunita in \cite{K1997}. We refer to various statements of this work \cite[Theorem 1.4.1, Theorem 3.4.1, Theorem 4.3.2., Theorem 4.6.5, Example on pages 106f.]{K1997} when we discuss conditions on the stochastic drift coefficients $\sigma^k$ such that we have sufficient control on the local in time dispersion.

\subsection{Structure of the paper}
In Section \ref{sec: notation} we give an overview of the relevant notation and solution concepts.
In Section \ref{sec: strichartz_stoch} we analyze stochastic kinetic transport and show Strichartz estimates for stochastic kinetic transport in the cases where Assumption \ref{ass:disp_local} or \ref{ass:disp_glob} are fulfilled.
In Section \ref{sec:dispersion} we discuss conditions and counterexamples for noise coefficients $\sigma^k$ that satisfy Assumption \ref{ass:disp_local} or Assumption \ref{ass:disp_glob}.
Finally, in Section \ref{sec:chemotaxis_det} we prove Theorem \ref{thm:main}. 

\section{Notation and preliminaries}\label{sec: notation}
When we consider global in time solutions in the stochastic setting, we usually work in $\Omega \times [0,T] \times \R^d \times \R^d$ for arbitrary but finite time $T > 0$. We will frequently omit the dependency on the probabilistic variable $\omega$.\\
We frequently use mixed Lebesgue spaces. We either work in the space $L_t^rL_x^pL_v^q:=L^r([0,T], L_x^pL_v^q)$ for arbitrary but fix $T$ or the space $L_t^rL_x^pL_v^q:=L^r([0,\tau], L_x^pL_v^q)$ for a stopping-time $\tilde{\tau}$ with $0 \leq \tau\leq T$. \\
Furthermore, in the proof of Theorem \ref{thm:main} we consider the Sobolev-space $$W_t^{\kappa, \lambda}([0,T],\R) = \left\{f \in L^\lambda([0,T],\R):
\int_0^{T} \int_0^{T}\frac{\left \lvert f(t)-f(s) \right\rvert ^{\lambda}}{|t-s|^{\kappa \lambda + 1}} \d s \d t < \infty
\right\}$$ and the semi-norm
$
\lVert f \rVert_{\dot{W}_t^{\kappa,\lambda}} =
\left(\int_0^{T} \int_0^{T}\frac{\left \lvert f(t)-f(s) \right\rvert ^{\lambda}}{|t-s|^{\kappa \lambda + 1}} \d s \d t\right)^{\frac{1}{\lambda}}.
$ 

For calculations it is often more convenient to work with the equation in It\^o form and use some abbreviations. Consequently, we will occasionally use the notation $a(x,v) = \frac{1}{2}\sum_{k\in \N} \sigma^k(x,v) \otimes \sigma^k(x,v)$, $\mathcal{L}_{\sigma}\varphi \defeq \div _v (a\nabla_v \varphi)$, $g \defeq \int_{V} K(S) f' - K^{\ast}(S) f \d v'$ and $g^n \defeq \int_{V} K^n(S^n) (f^n)' - (K^n)^{\ast}(S^n) f^n \d v'$.

We are interested in weak martingale solutions to \eqref{IVP_stoch}. 

\begin{definition}[Weak martingale solution]\label{def: weak martingale_sol}
	Fix $T \in (0, \infty)$. We say $f$ is a \textit{weak martingale solution to the stochastic model of chemotaxis (SCT) \ref{IVP_stoch} on $[0,\tau]$}
	provided there exists a stochastic basis 
	$(\Omega, \mathcal{F}, \mathbb{P}, (\mathcal{F}_t)_{t\ge 0}^{\tau}, (\beta^k)_{k \in \N})$ where $\tau\in [0,T]$ is a $\mathbb{P}$-almost surely positive $\{\mathcal{F}_t\}$-stopping-time such that
	\begin{enumerate}
		\item For all $\varphi \in C_c^{\infty}(\R^{2d})$, the process 	 $\left \langle f, \varphi \right\rangle: \Omega \times [0, \tau] \rightarrow \R$ admits $\mathbb{P}$-a.s. continuous sample paths. Moreover, $f$ belongs to $L^2(\Omega; L^{\infty}_t (L^1_{x,v}))$.
		\item $f(\cdot,\omega)$ is a nonnegative element of $L^r([0,\tau],L_x^pL_v^q)$, $\mathbb{P}$-a.s. 
		\item The processes $(f_t)^{\tau}_{t=0},(\int_0^t g_sds)^{\tau}_{t=0}$ and each Brownian motion $(\beta^k_t)^{\tau}_{t=0}$ are adapted to $(\mathcal{F}_t )^{\tau}_{t=0}$.
		\item For all test functions $\varphi \in C^{\infty}_c(\R^{2d} )$, the process $(M_t(\varphi))^{\tau}_{t=0}$ defined by 
		\begin{align*}
			M_t(\varphi) =	\int_{\R^{2d}} f_t\varphi dxdv - \int_{\R^{2d}}f_0\varphi dxdv - \int_{0}^{t}\int_{\R^{2d}}f(v \cdot \nabla_x \varphi + \mathcal{L}_{\sigma} \varphi) + g\varphi dxdvds
		\end{align*}
		is a $(\mathcal{F}_t)^{\tau}_{t=0}$ martingale. Moreover, its quadratic variation and cross variation with respect to each $\beta^k$ are given by
		\begin{align*}
			\langle \langle M(\varphi), M(\varphi) \rangle \rangle_t &=\sum_{k \in \N}\int_0^t\left(\int_{\R^{2d}}f_s\sigma^k \cdot \nabla_v \varphi \d x \d v\right)^2 ds,\\
			\langle \langle M(\varphi),\beta^k \rangle \rangle_t &= \int_0^t\int_{\R^{2d}} f_s\sigma^k \cdot \nabla_v \varphi \d x \d v.
		\end{align*}
	\end{enumerate}
\end{definition}

Let us define a set of parameters which allow for Strichartz estimates both in the deterministic and stochastic setting. 
\begin{definition}\label{def: admissible}
	A tuple $(q,r,p,a)$ is called \textit{admissible}, if
	\begin{align*}
		\begin{array}{lll} 				
			\frac{2}{r} = d\left(\frac{1}{q}-\frac{1}{p}\right), &\frac{1}{a} = \frac{1}{2} (\frac{1}{p} + \frac{1}{q}),&\\
			1 \leq a \leq \infty, & q^*(a) \leq q \le a, & a \leq p \leq p^*(a),
		\end{array}
	\end{align*}
	with 
	\begin{align*}
		\begin{cases}
			q^{\ast}(a) = \frac{da}{d+1}, \quad p^{\ast}(a) = \frac{da}{d-1},& \text{ if } \frac{d+1}{d} \le a \le \infty\\
			q^{\ast}(a) = 1, \quad p^{\ast}(a) = \frac{a}{2-a},& \text{ if } 1 \le a \le \frac{d+1}{d},
		\end{cases}
	\end{align*}
	except in the case $d = 1, (r, p, q) = (a,\infty, \frac{a}{2})$.
\end{definition}

\begin{definition}
	Two admissible pairs $(q,r,p,a)$ and $(\tilde{q},\tilde{r},\tilde{p},\tilde{a})$ are called \textit{jointly admissible}, if	$\tilde{a} = a'$.
\end{definition}

When examining the properties of the stochastic flow, we need to address the local characteristic $(a(x,y,t), b(x,t),A_t)$ of a semimartingale (compare \cite[pp. 84ff.]{K1997}).

\begin{definition}\label{def: local_characteristic_semi}
	Let $F(x,t)$ be a family of continuous semimartingales decomposed as $F(x,t)=M(x,t) + B(x,t)$, where $M(x,t)$ is a continuous localmartingale and $B(x,t)$ is a continuous process of bounded variation. Let $(a(x,y,t),A_t)$ the local characteristic of $M(x,t)$, which is defined via the joint characteristic of $M(x,t)$ and $M(y,t)$ (see \cite[pp. 79ff.]{K1997} for a precise definition). Assume that $B(x,t)$ is written as $$B(x,t)=\int_0^t b(x,x) \d A_s,$$ where $b(x,t)$ is a family of predictable processes. Then, the triple $(a(x,y,t), b(x,t),A_t)$ is called the \textit{local characteristic of the family of semimartingales $F(x,t)$}.
\end{definition}

 More precisely, we have to show that these characteristic belongs to the class $B_{ub}^{m,\delta}$ or $B_{b}^{m,\delta}$ (compare \cite[pp. 72ff.]{K1997}).

\begin{definition}\label{def: local_characteristic_class}
	Let $m \in \N$ and $0 \leq \delta \leq 1$. The local characteristic $(a(x,y,t), b(x,t),A_t)$ of a semimartingale belongs to the \textit{class $B_{ub}^{m,\delta}$} if and only if the norms $\norm{a(t)}_{m+\delta}^{\sim}$ and $\norm{b(t)}_{m+\delta}$ are bounded. For a precise definition of these norms see \cite[pp. 72ff. and pp. 334f.]{K1997}.
\end{definition}

\begin{definition}\label{def: local_characteristic_class_B_b}
	Let $m \in \N$, $0 \leq \delta \leq 1$. The local characteristic $(a(x,y,t), b(x,t),A_t)$ of a semimartingale belongs to the \textit{class $B_{b}^{m,\delta}$} if and only if the terms $\int_0^T\norm{a(t)}_{m+\delta}^{\sim}\d A_t < \infty$ and $\int_0^T\norm{b(t)}_{m+\delta}< \infty$ a.s., where the norms $\norm{a(t)}_{m+\delta}^{\sim}$ and $\norm{b(t)}_{m+\delta}$ are defined as in Definition \ref{def: local_characteristic_class}.
\end{definition}

\section{Stochastic kinetic transport and Strichartz estimates}\label{sec: strichartz_stoch}
Before showing the existence of a solution to the stochastic chemotaxis system, we will first establish some results for linear and nonlinear stochastic kinetic transport, particularly focusing on dispersion and Strichartz estimates. A crucial component of proving these results is a thorough understanding and control of the stochastic flow. Therefore, let $\Phi_{s,t}(x,v)$ denote the stochastic flow associated with the SDE (\ref{SDE: stoch. flow_intro}), where $t\mapsto \Phi_{s,t}(x,v) = (X_t, V_t)$ is the solution of (\ref{SDE: stoch. flow_intro}) with the initial condition $\Phi_{s,s} (x,v) = (x,v)$. Let $\Psi_{s,t}(x,v)$ be the inverse of $\Phi_{s,t}(x,v)$ such that $\Psi_{s,t}(x,v) = \Phi_{t,s}(x,v)$. When referring specifically to the position or velocity components, we use $\Phi_{s,t}(x,v)^{(1)}$ or $\Phi_{s,t}(x,v)^{(2)}$ for the position or velocity parts, respectively, and similarly $\Psi_{s,t}(x,v)^{(1)}$ or $\Psi_{s,t}(x,v)^{(2)}$ for the position or velocity components of the inverse. 

If Assumption \ref{ass: ex_flow} is fulfilled, we deduce that the stochastic flow exists, is unique and $\mathbb{P}$-almost surely volume-preserving.

\begin{lem}\label{lem: volumepreserving}
	Consider the SDE \eqref{SDE: stoch. flow_intro}. Assume that Assumption \ref{ass: ex_flow} is satisfied. Then, for all $s<t$ the stochastic flow $\Phi_{s,t}$ exists, is unique and is $\mathbb{P}$-almost surely volume preserving with $$\left|\det D \Phi_{s,t}^{-1} (x,v) \right | = \left| \det D \Phi_{s,t}(x,v)\right | = 1.$$
\end{lem}
\begin{proof}
	The stochastic flow exists globally and is unique due to Kunita \cite[Theorem 3.4.1]{K1997}. Details can be found also in \cite[Example on pages 106f.]{K1997}.
	Moreover it is volume-preserving due to Kunita (\cite[Theorem 4.3.2]{K1997}). The main idea is to use a stochastic analogue of Liouville's theorem.
\end{proof}

\subsection{Stochastic kinetic transport}
In this section, we represent the solution of linear and nonlinear stochastic kinetic transport equations with respect to the stochastic flow. 
\begin{lem}
Let $f_0:\R^{2d}\rightarrow \R$ be continuously differentiable. Then, for almost all $\omega$ the unique strong solution of the linear stochastic kinetic transport equation
	\begin{align*}
		\d t f(\omega, t,x,v) 
		+ v \cdot \nabla_x f(\omega, t,x,v) \d t 
		+ \div_v \sum_k f\sigma^k \circ \d \beta^k_t
		&= 0 \\
		f(\omega,0,x,v) &= f_0(x, v)
	\end{align*}
	is given by
	\begin{align*}
		f(t,x,v)=f_0\left(\Psi_{0,t}(x,v)\right).\\
	\end{align*}
\end{lem}

\begin{proof} Since (\ref{SDE: stoch. flow_intro}) is in Stratonovich-form the corresponding backward SDE remains the same. Let $(\Xi_t,\mathcal{V}_t) = \Psi_{0,t}(x,v)$ the solution of the backward SDE given by
	\begin{align*}
		\d \Xi_t &= -\mathcal{V}_t \d t,\\ 
		\d \mathcal{V}_t &= -\sum_{k \in \N} \sigma^k(\Xi_t, \mathcal{V}_t) \circ \d \beta^k_t,
		\end{align*}
with the initial condition $\Xi_t(t) = x$ and $\mathcal{V}_t(t) = v.$
Let $f(t,x,v) = f_0(\Xi_t,\mathcal{V}_t)$. Then, we have 
	\begin{align*}
		\d f &= \sum_{i=1}^{d}\left(\frac{\partial}{\partial x_i}f_0\right)(\Xi_t,\mathcal{V}_t)\cdot \d (\Xi_t)_i + \sum_{i=1}^{d}\left(\frac{\partial}{\partial v_i}f_0\right)(\Xi_t,\mathcal{V}_t)\cdot \d (\mathcal{V}_t)_i.
	\end{align*}
	Substituting $\d \Xi_t = -\mathcal{V}_t \d t$ and $\d \mathcal{V}_t = -\sum_{k \in \N} \sigma^k(\Xi_t, \mathcal{V}_t) \circ \d \beta^k_t$ we obtain that
\begin{align*}
		\d f& = -\sum_{i=1}^{d}\left(\frac{\partial}{\partial x_i}f_0\right)(\Xi_t,\mathcal{V}_t)\cdot (\mathcal{V}_t)_i \d t - \sum_{i=1}^{d}\left(\frac{\partial}{\partial v_i}f_0\right)(\Xi_t,\mathcal{V}_t)\cdot \sum_{k \in \N} \sigma^k(\Xi_t, \mathcal{V}_t) \circ \d (\beta^k_t)_i.
\end{align*}
Rewriting this, we get
	\begin{align*}
		\d f= -v \nabla_x f \d t - \sum_{k \in \N} \div_v\left(f \sigma^k(x,v) \circ \d \beta^k_t\right).
	\end{align*}
	Thus, $f(t,x,v)=f_0\left(\Psi_{0,t}(x,v)\right)$ solves the linear stochastic kinetic transport solution. To establish the uniqueness of a strong solution, we define the function $v(t) \defeq u(s+t,\Phi_{s,t}(x,v)) =u(s+t,X_t,V_t)$ for a given strong solution $u$ of the linear kinetic transport equation and $t > -s$. By performing a similar calculation as above, we show that $\frac{\partial v(t)}{\partial t} = 0 $, implying that $u(s+t,\Phi_{s,t}(x,v))$ is constant. Therefore, with the initial value given, the uniqueness follows.
\end{proof}

\begin{remark}
	This equation also admits an unique distributional solution $f\in L^p$ provided that $f_0$ is bounded in $L^p$.
\end{remark}

\begin{remark}
	A special case of this stochastic linear kinetic transport equation is given by
\begin{align*}
	\d f + v \nabla_x f \d t + \nabla_v f \circ \d \beta_t &= 0 \\
	f|_{t=0} &= f_0.
\end{align*}
The solution to this equation is 
\begin{align*}
	f(t,x,v,\omega)
	&= f_0 \left(x - \int_{0}^{t}\beta_s\d s -t (v -\beta_t),v-\beta_t\right)\\
	&= g\left(t, x - \int_{0}^{t}\beta_s\d s, v-\beta_t\right),						
\end{align*}
where $g$ is a solution of the deterministic kinetic transport equation
\begin{align*}
	\frac{\partial}{\partial t} g + v\nabla_x g & = 0 \\
	g(0,x,v) &= f_{0}(x,v). 
\end{align*}
In this special case, there is only additive noise in the phase space variables. Therefore, the dispersive behavior is not affected by the stochastic drift.
\end{remark}

\begin{lem}[Inhomogeneous stochastic transport equation]
	Let $f_0:\R^{2d}\rightarrow \R$ and $h(\omega,\dot):(0,\infty)\times\R^{2d} \rightarrow \R$ be continuously differentiable. Then, for almost all $\omega$ the unique strong solution of the inhomogeneous stochastic kinetic transport equation
	\begin{align*}
		\d f(\omega, t,x,v) 
		+ v \cdot \nabla_x f(\omega, t,x,v) \d t 
		+ \div_v \sum_k f\sigma^k \circ \d \beta^k_t
		&= h(\omega, t, x,v) \d t \\
		f(\omega,0,x,v) &= f_0(x, v)
	\end{align*}
	is given by 
	\begin{align*}
		f(\omega, t, x, v)
		= &f_0 \left(\Psi_{0,t}(x,v)\right) + \int_{0}^{t} 
		h\left( \omega, s,\Psi_{s,t}(x,v) \right) \d s.
	\end{align*}	
\end{lem}
\begin{proof} This result follows from the variation of constants formula. 
\end{proof}

\begin{remark}
	We call functions that satisfy this variation of constants formula mild solutions. 
\end{remark}

\subsection{Dispersion estimates for stochastic kinetic transport}
To show dispersion estimates, we need to examine the behavior of $\Phi_{s,t}$, particularly its derivative with respect to $v$. If we can ensure that Assumption \ref{ass:disp_local} or Assumption \ref{ass:disp_glob} is satisfied, we can establish dispersion estimates either locally or globally in time. With these estimates in hand, we will first show Strichartz estimates and then, in Section \ref{sec:dispersion} we will explore conditions on $\sigma^k $ such that Assumption \ref{ass:disp_local} or Assumption \ref{ass:disp_glob} is met. 

With Lemma \ref{lem: volumepreserving} and Assumption \ref{ass:disp_local} or Assumption \ref{ass:disp_glob} we are able to show local in time dispersion estimates. 
\begin{lem}[Dispersive decay]\label{lem:dispersivdecay}	\
		\begin{enumerate}\item\label{it:disp_stoch1} Let $1 \leq a \leq \infty$ and $f \in L_{x,v}^a$. Then, for all $s,t$ and almost all $\omega$ we have
		\begin{align*}
			\norm{f(\Phi_{s,t}^{-1})}_{L_{x,v}^a} &= 	\norm{f}_{L_{x,v}^a}.
		\end{align*}
		\item\label{it:disp_stoch2} 	Let $1 \leq q \leq p \leq \infty$, $f \in L_{x}^pL_v^q$ and $\tau$ as in Assumption \ref{ass:disp_local} or Assumption \ref{ass:disp_glob}. Then, there exists $C > 0$ such that for all $s,t$ with $\vert t-s \vert \leq \tau$ and almost all $\omega \in \Omega$ we have
		\begin{align*}
			\norm{f(\Phi_{s,t}^{-1})}_{L_x^pL_v^q} 
			\leq C \cdot \left\vert t-s \right\vert ^{-d(\frac{1}{q}-\frac{1}{p})} \norm{f}_{L_x^qL_v^{p}}.
		\end{align*}
		\end{enumerate}
\end{lem}	
\begin{proof}
		\eqref{it:disp_stoch1} Using a change of variables we can rewrite
		 \begin{align*}
		\norm{f(\Phi_{s,t}^{-1})}_{L_{x,v}^a}&= \left(\int_{\R ^d}\int_{\R ^d}\left \lvert f( \Phi_{s,t}^{-1}(x,v)) \right \rvert ^a 
		\d v \d x\right)^{\frac{1}{a}} \\&= \left(\int_{\R ^d}\int_{\R ^d}\left \lvert f(x,v)\right \rvert ^a \left \lvert \det D \Phi_{s,t}^{-1}(x,v)\right \rvert	\d v \d x\right)^{\frac{1}{a}}.
	\end{align*}
	Since $\Phi$ is $\mathbb{P}$-almost surely volume preserving, the determinant satisfies $\left \lvert \det D \Phi_{s,t}^{-1}(x,v)\right \rvert = 1$. Thus, we obtain that $\mathbb{P}$-a.s. 
		\begin{align*}
			\ \left(\int_{\R ^d}\int_{\R ^d}\left \lvert f(x,v)\right \rvert ^a \left \lvert \det D \Phi_{s,t}^{-1}(x,v)\right \rvert	\d v \d x\right)^{\frac{1}{a}}
			 = \norm{f}_{L_{x,v}^a}.
		\end{align*}
		\eqref{it:disp_stoch2} Now, let $q \geq 1$. It is enough to prove, that $\mathbb{P}$-a.s.	
		\begin{align}\label{it:disp_stoch2a}
			 \norm{f(\Phi_{s,t}^{-1})}_{L_x^{\infty}L_v^q} \leq C \cdot \vert t-s \vert^{-\frac{d}{q}}\norm{f}_{L_x^qL_v^{\infty}}.
			\end{align} 
			Then, interpolation of \eqref{it:disp_stoch1} with $a = 1 $ and (\ref{it:disp_stoch2a}) with $q=1$ yields	
			\begin{align}\label{it:disp_stoch2b}
		\norm{f(\Phi_{s,t}^{-1})}_{L_x^{p}L_v^1} \leq C \cdot \vert t-s \vert^{-d(1-\frac{1}{p})}\norm{f}_{L_x^1L_v^{p}}.
			\end{align}
		And finally, interpolation of (\ref{it:disp_stoch1}) with $a = p $ and (\ref{it:disp_stoch2b}) yields	
		\begin{align}
			\norm{f(\Phi_{s,t}^{-1})}_{L_x^{p}L_v^q} \leq C \cdot \vert t-s \vert^{-d(\frac{1}{q}-\frac{1}{p})} \norm{f}_{L_x^qL_v^{p}}.
		\end{align}
		It remains to show (\ref{it:disp_stoch2a}). First, bounding the left-hand-side by using the $L^{\infty}$-norm and using a change of variables yields
			\begin{align*}
				\Vert f(\Phi_{s,t}^{-1}(x,v))\Vert_{L_v^q}
				&\leq \left(\int_{\R ^d} \sup_{w \in \R^{d}}\left\lvert f(\Phi_{t,s}(x,v)^{(1)},w)\right\rvert^q dv \right)^{\frac{1}{q}}\\
				&= \left(\int_{\R ^d} \sup_{w \in \R^{d}}\left\lvert f(v,w)\rvert^q \lvert \det D_v (\Phi_{t,s}(x,v)^{(1)}) \right\rvert^{-1} dv \right)^{\frac{1}{q}}.
			\end{align*}		
		Using that $\left\lvert \det D_v (\Phi_{t,s}(x,v)^{(1)}) \right\rvert\geq C \lvert t-s \rvert$ for $\vert t-s\vert \leq \tau$ due to Assumption \ref{ass:disp_local} we get 
		\begin{align*}
			\Vert f(\Phi_{s,t}^{-1}(x,v))\Vert_{L_v^q}
			\leq C\vert t-s \vert ^{-\frac{d}{q}} \Vert f \Vert_{L_x^q L_v^{\infty}}.
		\end{align*}
\end{proof}

\subsection{Strichartz estimates for stochastic kinetic transport}
\begin{lem}[Strichartz estimates]\label{lm: stoch_strich}
	Let $(q,r,p,a)$ and $(\tilde{q},\tilde{r},\tilde{p},a')$ two jointly admissible tuples, $\tau$ as in Assumption \ref{ass:disp_local} or Assumption \ref{ass:disp_glob} and $\tilde{\tau}\leq T$. Then, there exists $C 		\left(\left\lceil\frac{\tilde{\tau}}{\tau}\right\rceil \right)> 0$ depending only on the number of intervals with length up to $\tau$ between $0$ and $\tilde{\tau}$ such that for almost all $\omega$
	\begin{enumerate}
		\item \label{eq_hom lm: stoch_strich} the homogeneous part of the stochastic kinetic transport equation satisfies
		\begin{align*}
			\norm{f(\Psi_{0,t}(x,v))}_{L^r([0,\tilde{\tau}],L_x^pL_v^q)} &\leq C
			\left(\left\lceil\frac{\tilde{\tau}}{\tau}\right\rceil \right)
			\norm{f}_{L_{x,v}^a},
		\end{align*}
		\item \label{eq_inhom lm: stoch_strich}	the inhomogeneous part of the stochastic kinetic transport equation satisfies
		\begin{align*}
			\norm{\int_{0}^t f (s,\Psi_{s,t}(x,v))\d s}_{L^r([0,\tilde{\tau}],L_x^pL_v^q)} \leq &
			C\left(\left\lceil\frac{\tilde{\tau}}{\tau}\right\rceil \right)
			\dorm{f}{\tilde{r}'}{\tilde{p}'}{\tilde{q}'}.
		\end{align*}
	\end{enumerate}
\end{lem}

\begin{proof}
	(\ref{eq_hom lm: stoch_strich}): The main difference to the deterministic case is, that we have local in time dispersion in comparison with a fix time horizon. 
	By duality, we rewrite
	\begin{align*}
		\norm{f(\Phi_{0,t}^{-1}(x,v))}_{L^r([0,\tilde{\tau}],L_x^pL_v^q)}
		=	\sup_{\dorm{\phi}{r'}{p'}{q'} \leq 1} \int_{[0,\tilde{\tau}]} 
		\int_{\R^{d}}\int_{\R^{d}}f(\Phi_{0,t}^{-1}(x,v)) \phi(t,x,v) \d v \d x \d t.
	\end{align*}
	Now, let $k(\omega) \in \N_{0}$ such that $k\tau < \tilde{\tau} \leq (k+1)\tau$. Splitting the integral and using change of variables we rewrite
		\begin{align*}
		&\int_{[0,\tilde{\tau}]} 
		\int_{\R^{d}}\int_{\R^{d}}f\left(\Phi_{0,t}^{-1}(x,v)\right) \phi(t,x,v) \d v \d x \d t\\
		&= \sum_{m=0}^{k}\int_{[m\tau,(m+1)\tau]} 
		\int_{\R^{d}}\int_{\R^{d}}f\left(\Phi_{0,t}^{-1}(x,v)\right) \phi(t,x,v) \d v \d x \d t\\
		&=\sum_{m=0}^{k}\int_{[m\tau,(m+1)\tau]} 
		\int_{\R^{d}}\int_{\R^{d}}f\left(\Phi_{m\tau,0}(\Phi_{t,m \tau}(x,v))\right) \phi(t,x,v) \d v \d x \d t\\
		&=\sum_{m=0}^{k}\int_{[m\tau,(m+1)\tau]} 
		\int_{\R^{d}}\int_{\R^{d}}f\left(\Phi_{m\tau,0}(x,v)\right) \phi(t,\Phi_{m\tau,t}(x,v)) \d v \d x \d t.
	\end{align*}
	Moreover, by applying Hölder's inequality and the fact that $\Phi$ is volume preserving, we get
	\begin{align*}
		&\int_{[m\tau,(m+1)\tau]} 
		\int_{\R^{d}}\int_{\R^{d}}f\left(\Phi_{m\tau,0}(x,v)\right) \phi\left(t,\Phi_{m\tau,t}(x,v)\right) \d v \d x \d t\\	
		\leq &\norm{f\left(\Phi_{m\tau,0}(x,v)\right)}_{L_{x,v}^a} \cdot \norm{\int_{[m\tau,(m+1)\tau]}\phi\left(t,\Phi_{m\tau,t}(x,v)\right) \d t}_{L_{x,v}^{a'}}\\
		= & \norm{f}_{L_{x,v}^a} \cdot \norm{\int_{[m\tau,(m+1)\tau]}\phi\left(t,\Phi_{m\tau,t}(x,v)\right) \d t}_{L_{x,v}^{a'}}.
	\end{align*}
	With $a=\frac{2pq}{p+q}$ we find that $a'=\frac{2q'p'}{p'+q'}$. Therefore, with change of variables and Hölder's inequality we obtain
	\begin{align*}
		 \norm{\int_{[m\tau,(m+1)\tau]}\phi(t,\Phi_{m\tau,t}(x,v)) \d t}_{L_{x,v}^{a'}}^2 \leq & \int_{[m\tau,(m+1)\tau]^2}\norm{\phi(t,x,v)\phi(s,\Phi_{s,t})}_{L_{x,v}^{\frac{q'p'}{p'+q'}}} \d t \d s\\
		\leq &	\int_{[m\tau,(m+1)\tau]^2}\zorm{\phi(t)}{p'}{q'} \zorm{\phi\left(s,\Phi_{s,t}\right)}{q'}{p'} \d t \d s.
	\end{align*}
	With Lemma \ref{lem:dispersivdecay} the norm $\zorm{\phi\left(s,\Phi_{s,t}\right)}{q'}{p'}$ can be bounded by $C\vert t-s \vert ^{-d\left(\frac{1}{q}-\frac{1}{p}\right)}$. Finally, applying the Hardy-Littlewood-Sobolev inequality, we have
	\begin{align*}
	\norm{\int_{[m\tau,(m+1)\tau]}\phi(t,\Phi_{m\tau,t}(x,v)) \d t}_{L_{x,v}^{a'}}^2 \leq & C\cdot \dorm{\phi}{r'}{p'}{q'}^2.
	\end{align*}
	Summing over m gives the desired result.

	(\ref{eq_inhom lm: stoch_strich}): Using duality, change of variables and Hölder's inequality we get
	\begin{align*}
		&\left\Vert \int_{0}^{t} f\left(s, \Phi_{s,t}^{-1}(x,v)\right) \d s \right\Vert_{L_t^r L_x^p L_v^q}\\
		&\le \sup_{\Vert \Phi\Vert_{L_t^{r'} L_x^{p'} L_v^{q'}}\le 1}	
		\left\Vert 
		\int_{[0,\tilde{\tau}]}\Phi\left(t, \Phi_{0,t}(x,v)\right) \d t
		\right\Vert_{L_{x,v}^{a'}}
		\cdot	\left\Vert \int_{[0,\tilde{\tau}]} f(s, \Phi_{0,s}(x,v)) \d s \right\Vert_{L_{x,v}^{a}}.
	\end{align*}
	By duality and (\ref{eq_hom lm: stoch_strich}) we calculate
	\begin{align*}
		\left \lVert\int_{[0,\tilde{\tau}]} \Phi\left(t, \Phi_{0,t}(x,v)\right)\ \d t \right\rVert_{L_{x,v}^{a'}}
		\le 	&\sup_{\Vert\Psi\Vert_{L_{x,v}^{a}}\le 1} 
		\Vert \Phi \Vert_{L_t^{r'} L_x^{p'} L_v^{q'}}
		\cdot
		\left\Vert \Psi(t, \Phi_{0,t}^{-1}, x, v) \right\Vert_{L_t^r L_x^p L_v^q}\\
		\le	& C\left(\left\lceil\frac{\tilde{\tau}}{\tau}\right\rceil \right)\Vert \Phi \Vert_{L_t^{r'} L_x^{p'} L_v^{q'}},
	\end{align*}
	and 
	\begin{align*}
		\left\lVert\int_{[0,\tilde{\tau}]} f\left(s, \Phi_{0,s}(x,v)\right)\ \d s	\right\rVert_{L_{x,v}^{a}}
		&\le	\sup_{\Vert\theta\Vert_{L_{x,v}^{a'}}\le 1} 
		\Vert f\Vert_{L_t^{\tilde{r}'} L_x^{\tilde{p}'} L_v^{\tilde{q}'}}
		\cdot
		\left\Vert 
		\theta(s, \Phi_{0,s}^{-1}, x, v) 
		\right\Vert
		_{
			L_t^{\tilde{r}} L_x^{\tilde{p}} L_v^{\tilde{q}}			}\\
		&\leq C
		\left(\left\lceil\frac{\tilde{\tau}}{\tau}\right\rceil \right)\Vert f\Vert_{L_t^{\tilde{r}'}L_x^{\tilde{p}'}L_v^{\tilde{q}'}}.
	\end{align*} 
This implies
	\begin{align*}
		\left\Vert \int_{0}^{t} f\left(s, \Phi_{s,t}^{-1}(x,v)\right) \d s \right\Vert_{L_t^r L_x^p L_v^q}	
		\le	 C\left(\left\lceil\frac{\tilde{\tau}}{\tau}\right\rceil \right)\Vert f \Vert_{L_t^{\tilde{r}'} L_x^{\tilde{p}'} L_v^{\tilde{q}'}}.
	\end{align*}
\end{proof}

\section{Different types of external random force}\label{sec:dispersion}
In this section, we will first examine conditions under which Assumption \ref{ass:disp_local} is satisfied for the stochastic drift driven by $\sigma^k$. Subsequently, we will provide several examples and counterexamples of functions $\sigma^k$ to illustrate scenarios where dispersion holds globally in time, i.e., where Assumption \ref{ass:disp_glob} is fulfilled. Consider the solution $\Phi_{s,t}(x,v) = (X_t,V_t)$ of \eqref{SDE: stoch. flow_intro} with initial condition $X_s = x$ and	$V_s = v$. Our aim is to analyze the dispersive behavior of the stochastic flow $\Phi_{s,t}(x,v) = (X_t,V_t)$.

\subsection{External random force that allows for local in time dispersion}\label{subsec:dispersion_local}
In this section, we will show that under suitable regularity conditions on $\sigma^k$ Assumption \ref{ass:disp_local} is satisfied. 
\begin{lem}\label{lem: assumption_local_disp}
	Let $\sigma^k \in C^3(\R^{2d})$ with $\div_v \sigma^k = 0$ for all $k$ such that $$\sum_k\left(\sum_{|\alpha| \leq 3} \left\lVert D_v^{\alpha}\sigma^k\right\rVert_{\infty} +\sum_{|\alpha| \leq 2} \left\lVert D_x^{\alpha}\sigma^k\right\rVert_{\infty}\right)< \infty.$$ Then, Assumption \ref{ass:disp_local} is fulfilled. 
\end{lem}
\begin{proof}
	Our strategy of the proof is a perturbative approach since we know that the deterministic equation always satisfies Assumption \ref{ass:disp_glob} and consequently, Assumption \ref{ass:disp_local}. 
	Since the SDE \eqref{SDE: stoch. flow_intro} is in Stratonovich-form the backward and forward stochastic flow are the same. However,
 it is more convenient to work with the equation in It\^o formulation. Thus, the integral equation is given by
		\begin{align*}
		\begin{pmatrix}
			X_t\\
			V_t
		\end{pmatrix}
		=& 		\begin{pmatrix}
			x\\
			v
		\end{pmatrix}
		+ \int_s^t\begin{pmatrix}
			V_r\\
			0
		\end{pmatrix}
		\d r + \sum_{k}\int_s^t 
		\begin{pmatrix}
			0\\
			\sigma^k(X_r, V_r)
		\end{pmatrix}
		\circ \d \beta^k\\
		=& 		\begin{pmatrix}
			x\\
			v
		\end{pmatrix}
		+ \int_s^t\begin{pmatrix}
			V_r\\
			0
		\end{pmatrix}
		\d r + \sum_{k}\int_s^t 
		\begin{pmatrix}
			0\\
			\sigma^k(X_r, V_r)
		\end{pmatrix}
		\d \beta^k_r\\ 
		&+ \frac{1}{2} \sum_k \int_s^t \begin{pmatrix}
			0\\
			\sum_{j=1}^d\begin{pmatrix}
				\frac{\partial \sigma^k_i(X_r, V_r)}{\partial v_j}\sigma^k_j (X_r, V_r)
			\end{pmatrix}_{i=1,\dots,d}
		\end{pmatrix}
		\d r.
	\end{align*}
	
	Thus, defining $\Phi_{s,t}(x,v):= (X_t,V_t)$ we rewrite,
	\begin{align*}
		\Phi_{s,t}(x,v)=& 		\begin{pmatrix}
			x + (t-s) v\\
			v
		\end{pmatrix}
		+ \begin{pmatrix}
			\int_s^t \sum_{k}\int_s^r 
			\sigma^k(\Phi_{s,u}(x,v))
			\d \beta^k_u \d r\\
			0
		\end{pmatrix}\\
		&+ \begin{pmatrix}
			\int_s^t \frac{1}{2} \sum_k \int_s^r \sum_{j=1}^d\begin{pmatrix}
				\frac{\partial \sigma^k_i \left(\Phi_{s,u}(x,v)\right)}{\partial v_j}\sigma^k_j\left(\Phi_{s,u}(x,v)\right)
			\end{pmatrix}_{i=1,\dots,d}	\d u \d r\\
			0
		\end{pmatrix}\\
		& + \sum_{k}\int_s^t 
		\begin{pmatrix}
			0\\
			\sigma^k\left(\Phi_{s,r}(x,v)\right)
		\end{pmatrix}
		\d \beta^k_r\\
		&+ \frac{1}{2} \sum_k \int_s^t \begin{pmatrix}
			0\\
			\sum_{j=1}^d\begin{pmatrix}
				\frac{\partial \sigma^k_i\left(\Phi_{s,r}(x,v)\right)}{\partial v_j}\sigma^k_j\left(\Phi_{s,r}(x,v)\right)
			\end{pmatrix}_{i=1,\dots,d}
		\end{pmatrix}
		\d r.\\
	\end{align*}
	First, we aim to demonstrate that $\Phi_{s,t}$ is continuously differentiable with respect to the spatial parameters $x$ and $v$ for all $s,t$. Afterwards, applying Kolmogorov's continuity theorem (\cite{K1997} Theorem 1.4.1), we will show that all resulting terms are $\alpha-$Hölder-continuous for $\alpha < \frac{1}{4}$. Finally, using the fact that the determinant is continuous, we will obtain the desired result.\\
According to Kunita \cite[ Theorem 4.6.5]{K1997} the stochastic flow $\Phi_{s,t}$ is continuously differentiable with respect to the spatial variables $x$ and $v$, given that its local characteristic belongs to the class $B_{ub}^{1,1}$. To be specific, let us consider the local characteristic $(a(x,v,y,u),b(x,v),A_t)$ of the semimartingale 
	\begin{align*}
			F(t,x,v)
			= 	
			\int_s^t\begin{pmatrix}
					v\\
					\frac{1}{2} \displaystyle{\sum_k	\sum_{j=1}^d}
					\begin{pmatrix}
							\frac{\partial \sigma^k_i(x,v)}{\partial v_j}\sigma^k_j(x,v)\\
						\end{pmatrix}_{i=1,\dots,d}
				\end{pmatrix}
			\d r
			+ \sum_{k}\int_s^t 
			\begin{pmatrix}
					0\\
					\sigma^k(x,v)
				\end{pmatrix}
			\d \beta^k_r.
		\end{align*}
This local characteristic is defined by
	\begin{align*}
		&(a(x,v,y,u),b(x,v),A_t) 
		=\\	& 
		\left[\begin{pmatrix}
			&0&0\\
			&0&\begin{pmatrix}
				\displaystyle{\sum_{k}}\sigma^k_i(x,v)\sigma^k_j(y,u)
			\end{pmatrix}_{i,j \in \{1,\dots, d\}}
		\end{pmatrix},
		\begin{pmatrix}
			v\\
			\frac{1}{2} 
			\displaystyle{\sum_k \sum_{j=1}^d}
			\begin{pmatrix}
				\frac{\partial \sigma^k_i(x,v)}{\partial v_j}\sigma^k_j(x,v)\\
			\end{pmatrix}_{i=1,\dots,d}
		\end{pmatrix},t 
		\right].
	\end{align*}
	It belongs to the class $B_{ub}^{1,1}$ because the norms $\norm{a}_{1+1}^{\sim} $ and $\norm{b}_{1+1}$ are bounded. Specifically, 
we obtain
	\begin{align*}
		\norm{a}_{1}^{\sim}
		\leq 	& \sum_k \max_{i,j}\sup_{(x,v),(y,u)}\frac{\left \vert \sigma^k_i (x,v)\sigma^k_j(y,u) \right \vert}{(1+\vert (x,v)\vert)(1+\vert (y,u)\vert)}\\
		&+ \sum_k \max_{i,j}\sup_{(x,v),(y,u)}\sum_{\vert \alpha \vert = 1 }\left \vert D_{x,v}^{\alpha}\sigma^k_i (x,v)D_{y,u}^{\alpha}\sigma^k_j(y,u) \right \vert\\
		\leq & \sum_k \max_{i} \sup_{(x,v)}\left \vert \sigma^k_i(x,v) \right \vert^2
		+ \sum_k \max_{i} \sup_{(x,v)} \sum_{\vert \alpha \vert = 1 } \left \vert
		D_{x,v}^{\alpha}\sigma^k_i (x,v) \right \vert^2. 
	\end{align*}
Moreover,
we calculate 
	\begin{align*}
		&\norm{D_{x,v}^{\alpha}D_{y,u}^{\alpha} a }_1^{\sim} \\&= \sum_k \sup_{(x,v)\neq (x',v'),(y,u) \neq (y',u')} \frac{\left \vert D^{\alpha}_{x,v}\sigma^k_i(x,v)
			-D^{\alpha}_{x,v}\sigma^k_i(x',v')\right \vert \left \vert D^{\alpha}_{y,u}\sigma^k_i(y,u)
			-D^{\alpha}_{y,u}\sigma^k_i(y',u')\right \vert}{\vert (x,v)-(x',v')\vert \vert (y,u)-(y',u')\vert }\\
		&\leq \sum_k\left( \max_i \sup_{(x,v)\neq (x',v')}\frac{\left \vert D^{\alpha}_{x,v}\sigma^k_i(x,v)
			-D^{\alpha}_{x,v}\sigma^k_i(x',v')\right \vert}{\vert (x,v)-(x',v')\vert}\right)^2.
	\end{align*}
For the norm $\norm{b}_{1}$ we get
	\begin{align*}
		\norm{b}_{1}
		\leq&\sum_k \sum_j \left(\max_i \sup_{(x,v)}\left \vert \frac{\partial \sigma^k_i(x,v)}{\partial v_j} \right\vert \right)\left(\max_i\sup_{(x,v)} \vert \sigma^k_i \vert\right) \\
		&+ \sum_k \sum_j \sum_{\vert \alpha \vert = 1 } \left(\max_i\sup_{(x,v)} \left\vert D^{\alpha}\frac{\partial \sigma^k_i}{\partial v_j}\right \vert\right)\left(\max_i\sup_{(x,v)} \left\vert \sigma^k_i (x,v)\right\vert\right)\\
		& + \sum_k \sum_j \sum_{\vert \alpha \vert = 1 } \left(\max_i\sup_{(x,v)}
		\left\vert \frac{\partial \sigma^k_i }{\partial v_j }\right\vert\right)\left(\max_i \sup_{(x,v)} \left\vert D^{\alpha}\sigma^k_i (x,v)\right\vert\right), 
	\end{align*}
		and we obtain 
	\begin{align*}
			\frac{\vert D^{\alpha} b(x,v) - D^{\alpha}b(y,u)\vert}{\vert (x,v) - (y,u)\vert} \leq &\left\vert D^{\alpha}\left(\frac{\partial \sigma^k_i(x,v)}{\partial_{v_j}}\right)\right\vert \frac{\vert \sigma^k_j(x,v)- \sigma^k_j (y,u)\vert}{\vert(x,v)-(y,u)\vert} \\
		&+ \left\vert \sigma^k_j (y,u)\right \vert \frac{\left \vert D^{\alpha}\left(\frac{\partial \sigma^k_i(x,v)}{\partial{v_j}}\right)-D^{\alpha}\left(\frac{\partial \sigma^k_i(y,u)}{\partial{u_j}}\right)\right\vert }{{\vert (x,v) - (y,u)\vert}}\\
		&+ \left\vert \left(\frac{\partial \sigma^k_i(x,v)}{\partial{v_j}}\right)\right\vert \frac{\vert D^{\alpha}\sigma^k_j(x,v)- D^{\alpha}\sigma^k_j (y,u)\vert}{\vert(x,v)-(y,u)\vert} \\
		&+ \left \vert D^{\alpha}\sigma^k_j (y,u)\right\vert\frac{\left\vert\left(\frac{\partial \sigma^k_i(x,v)}{\partial{v_j}}\right)-\left(\frac{\partial \sigma^k_i(y,u)}{\partial{u_j}}\right)\right\vert }{{\vert (x,v) - (y,u)\vert}}.
	\end{align*}
	All these terms are bounded because $\sigma^k$ and its derivatives are bounded and Lipschitz continuous. Therefore, by Kunita's theorem \cite[Theorem 4.6.5]{K1997} $\Phi_{s,t}$ is continuously differentiable with respect to the spatial parameters $x,v$ for all $s,t$. \\
	Let $i=1, \dots, 2d$, $l=1,\dots d$ and $y_m := \begin{cases} x_m & m \leq d \\ v_{m-d} & m > d \end{cases}.$ Then, the differentiation is given by
\begin{align}\label{lem: assumption_local_disp:eq:differentiation}
		\frac{\partial \Phi_{s,t}^i(x,v)}{\partial v_l} =
		\begin{cases}
		 (t-s)\left(\delta_{li} + \frac{1}{t-s}\int_s^t\Delta_v(l,i)\right)\d u & i =1,\dots d \\
		 \delta_{li} + \Delta_v(l,i) & i=d+1\dots 2d\\
		\end{cases}
\end{align}
	with remainders
	\begin{align*}
		&\Delta_v(l,i)\\ &= \sum_k \sum_{m=1}^{2d}\int_s^u \frac{\partial\sigma^k_i}{\partial y_m}(\Phi_{s,r})\frac{(\partial \Phi_{s,r})_m}{\partial v_l}(x,v) \d \beta^k_r
		\\ 
		&+ \frac{1}{2}\sum_k\sum_{j=1}^{d}\sum_{m=1}^{2d}\int_s^u \left(\left(\frac{\partial^2\sigma^k_i}{\partial y_m\partial v_j}(\Phi_{s,r}) \sigma^k_j(\Phi_{s,r})+\frac{\partial\sigma^k_i}{\partial v_j}(\Phi_{s,r})\frac{\partial \sigma^k_j}{\partial y_m}(\Phi_{s,r})\right)\frac{(\partial \Phi_{s,r})_m}{\partial v_l}(x,v)\right)\d r .
	\end{align*}
	To bound the stochastic integrals from above we use the Burkholder-Davis-Gundy inequality combined with the regularity assumptions on $\sigma^k$.  More precisely, there exists $C$ independent of x,v such that for all s,u
	\begin{align*}
		&\mathbb{E}\left(\left |\int_s^u\frac{\partial\sigma^k_i}{\partial y_m}(\Phi_{s,r})\frac{(\partial \Phi_{s,r})_m}{\partial v_l}(x,v) \d \beta^k_r\right |^4\right)\\
		&\leq \mathbb{E}\left(\left(\int_s^u\left(\frac{\partial\sigma^k_i}{\partial y_m}(\Phi_{s,r})\frac{(\partial \Phi_{s,r})_m}{\partial v_l}(x,v) \right)^2\d r\right)^2\right)\\
		&\leq (u-s)^2 \mathbb{E}\left(\sup_{r\in [s,u]}\left \lVert \frac{\partial\sigma^k_i}{\partial y_m}(\Phi_{s,r})\right \rVert_{\infty}^4 \left \lVert \frac{(\partial \Phi_{s,r})_m}{\partial v_l}(x,v)\right \rVert_{\infty}^4 \right) \leq C (u-s)^2
	\end{align*}
	and analogously there exists $C$ independent of x,v such that for all s,u
	\begin{align*}
		&\mathbb{E}\left(\left |\int_s^u\frac{\partial\sigma^k_i}{\partial y_m}(\Phi_{s,r})\frac{(\partial \Phi_{s,r})_m}{\partial x_l}(x,v) \d \beta^k_r\right |^4\right)
		\leq C (u-s)^2.
	\end{align*}
Thus, by Kolmogorov's continuity theorem, we establish that all stochastic integrals are $\alpha-$Hölder continuous with $\alpha < \frac{1}{4}$. Therefore, there exists $C_1(\omega)$ such that all stochastic terms are bounded by $C_1(\omega)|t-s|^{\alpha}.$
Thanks to the regularity assumptions on $\sum_k \sigma^k$ we can also handle the deterministic part. More precisely, there exists $C_2(\omega)$ such that the deterministic part is bounded by $C_2(\omega)|t-s|$.
Thus, there exists a constant $C(\omega)$, such that  
	\begin{align}\label{lem: assumption_local_disp:eq:hoelder_continuity}
		\lVert \Delta_v(l,i)\rVert \leq C(\omega) |t-s|^{\alpha} + C(\omega)|t-s|.
	\end{align}
Therefore, for $|t-s|$ small enough, where the smallness might depend on $C(\omega)$ the values of the remainder $\lVert \Delta_v\rVert$ becomes arbitrary small. Thus, by \cite[Theorem 8.1]{MN1988} $\det(E_d+\Delta_v)$ can be approximated by $1+\tr{(\Delta_v)}$.
Consequently, there exists a constant $C$ independent of $\omega$  and a $\mathbb{P}$-almost surely positive stopping-time $\tau(\omega)$ such that for all $|t-s| \leq \tau$ we have
	\begin{align*}
		\left|\det D_v \Phi_{s,t}(x,v)^{(1)}\right|& = \left|\det (t-s)\left(E_d+\Delta_v\right)\right| \geq C |t-s|^d. 
	\end{align*} 
\end{proof}

\begin{remark}\label{rem:large_dev}
 Kolmogorov's continuity theorem gives in addition, that for all $\delta > 0$ there exists a deterministic constant $K$ such that 
	\begin{align*}
		\mathbb{P}\left(\left \vert\int_s^u\frac{\partial\sigma^k_i}{\partial y_m}(\Phi_{s,r})\frac{(\partial \Phi_{s,r})_m}{\partial v_l}(x,v) \d \beta^k_r\right \vert \leq K \vert u-s \vert ^{\alpha}\right) \geq 1-\delta.
	\end{align*}
Consequently, there exists a constant $C$ and $\tau$ independent of $\omega$ such that we conclude
	\begin{align*}
		\mathbb{P}\left(\left|\det D_v \Phi_{s,t}(x,v)^{(1)}\right|\geq C |t-s|^d, \text{ for all } \vert t - s\vert \leq \tau\right) \geq 1-\delta.
	\end{align*}
\end{remark}

\begin{remark}
	The above lemma indicates that we require the drift coefficients $\sigma^k$ to be bounded in $x$ and $v$ up to the second derivative, with Lipschitz continuity in $v$. However, it may be sufficient to consider coefficients $\sigma^k \in C^{2+\varepsilon }$ for $\varepsilon >0$, to ensure that Assumption \ref{ass:disp_local} is satisfied.
\end{remark}

\subsection{External random force that allows for global in time dispersion}\label{subsec: dispersion_glob}
Under special conditions on the coefficients $\sigma^k$ the stochastic perturbation does not influence the dispersive character of the deterministic kinetic transport equation. We present here some classes of affine linear functions $\sigma^k$ that satisfy Assumption \ref{ass:disp_glob}.

\begin{lem}\label{ex:glob_disp}
		Let $\sigma^k \in C^1(\R^{2d})$ with $\div_v \sigma^k = 0$ for all $k$ such that $$\sum_k\left(\sum_{|\alpha| = 1} \left\lVert D_v^{\alpha}\sigma^k\right\rVert_{\infty}\right)< \infty.$$ If furthermore, one of the following conditions is fulfilled, then Assumption \ref{ass:disp_glob} is satisfied for all $\tau \in (0,\infty)$. \begin{enumerate}
		\item\label{it:glob_disp1} $ \sigma^k (x,v) = \sigma^k(v)$ is affine linear and $\exists N \in \N$ such that $\sigma^k \equiv 0$ for all $k$ greater than $N$ and 
		\begin{align*}
			\Sigma_2& = 		\begin{pmatrix}		
				\partial_{v^1} \sum_{k=1}^{N} \sigma^k_1(\Phi_{s,t}(x,v)) \circ d\beta^k_t 	& \cdots& \partial_{v^d} \sum_{k=1}^{N} \sigma^k_1(\Phi_{s,t}(x,v)) \cdot \circ d\beta^k_t	 \\
				\vdots	& \ddots	& \vdots 		\\
				\partial_{v^1} \sum_{k=1}^{N} \sigma^k_d(\Phi_{s,t}(x,v))\circ d\beta^k_t	& \cdots& \partial_{v^d} \sum_{k=1}^{N} \sigma^k_d(\Phi_{s,t}(x,v)) \cdot \circ d\beta^k_t 
			\end{pmatrix}\\
			& = 		\begin{pmatrix}		
				\sum_{k=1}^{N} \mu_{11}^{(k)}\circ d\beta^k_t 	& \cdots& \sum_{k=1}^{N} \mu_{1d}^{(k)}\circ d\beta^k_t	 \\
				\vdots	& \ddots	& \vdots 		\\
				\sum_{k=1}^{N} \mu_{d1}^{(k)}\circ d\beta^k_t	& \cdots& \sum_{k=1}^{N} \mu_{dd}^{(k)} \circ d\beta^k_t 
			\end{pmatrix}
		\end{align*} is nilpotent.
			\item\label{it:glob_disp2} $ \sigma^k (x,v) = \sigma^k(v)$ is affine linear and $\Sigma_2^{(k)}$ is diagonal for all $k$. 
			\item\label{it:glob_disp3} $ \sigma_1 (x,v) = \sigma_1(v)$ is affine linear and $\sigma^k \equiv 0$ for all $k$ greater than 1 and 
		\begin{align*}
			\Sigma_2& = 	\begin{pmatrix}		
				\mu_{11}\circ d\beta^1_t 	& \cdots& 	\mu_{1d}\circ d\beta^1_t 	 \\
				\vdots	& \ddots	& \vdots 		\\
				\mu_{d1}\circ d\beta^1_t 	& \cdots& 	\mu_{dd}\circ d\beta^1_t 
			\end{pmatrix}
		\end{align*} is equivalent to a Matrix in Jordan normal form with real eigenvalues.
		\end{enumerate}
\end{lem}

\begin{proof}
	In all the cases stated here we calculate the stochastic flow explicitly. Since $\sigma^k(x,v) = \sigma^k(v)$ is affine linear for all $k$ we have to solve
	\begin{align}\label{ex:glob_disp_flow}
		\begin{split}
					\d X_t &= V_t \d t,\\
		\d V_t &= \sum_k \underbrace{
			\begin{pmatrix}
			 \mu_{11}^{(k)} \circ \d \beta^k & \dots & \mu_{1d}^{(k)} \circ \d \beta^k \\
				\vdots & \ddots & \vdots \\
				 \mu_{d1}^{(k)} \circ \d \beta^k & \dots & \mu_{dd}^{(k)} \circ \d \beta^k
			\end{pmatrix}
		}_{=:\Sigma_2^{(k)}} V_t
		+
		\begin{pmatrix}
			\sum_{k} c_{1}^{(k)} \circ \d \beta^k \\
			\vdots \\
			\sum_{k} c_{d}^{(k)} \circ \d \beta^k
		\end{pmatrix}
	\end{split}.
	\end{align}	
	First, note that the corresponding local characteristic is given by 	\begin{align*}
		&(a(x,v,y,u),b(x,v),A_t) 
		=\\	& 
		\left[\begin{pmatrix}
			&0&0\\
			&0&\begin{pmatrix}
				\displaystyle{\sum_{k}}\displaystyle{\sum_{l=1}^d}\mu^{(k)}_{il}v_l\displaystyle{\sum_{m=1}^d}\mu^{(k)}_{jm}u_m
			\end{pmatrix}_{i,j \in \{1,\dots, d\}}
		\end{pmatrix},
		\begin{pmatrix}
			v\\
			\frac{1}{2} 
			\displaystyle{\sum_k \sum_{j=1}^d}
			\begin{pmatrix}
			 \mu^{(k)}_{ij}\displaystyle{\sum_{m=1}^d}\mu^{(k)}_{jm}v_m\\
			\end{pmatrix}_{i=1,\dots,d}
		\end{pmatrix},t 
		\right].
	\end{align*}
	With the boundedness condition and the fact, that either $\sigma^k \neq 0 $ is only true for a finite number of $k$ or $\Sigma_2^k$ is not diagonal, this local characteristic belongs to the class $B_b^{0,1}$.
	Since the matrix-exponential can be calculated in this cases we know that the solution $\Phi_{s,t}(x,v) = (X_t,V_t)$ of the linear inhomogeneous SDE \eqref{ex:glob_disp_flow} is given by
	\begin{align*}
		X_t
		=&x + 	\int_{s}^{t} e^{\sum_{k} (\beta^k_u-\beta_s^k) \Sigma_2^{(k)}} \d u \cdot v \\
		&+	\int_{s}^{t} e^{\sum_{k} (\beta^k_u-\beta_s^k) \Sigma_2^{(k)}} 
		\cdot \int_{s}^{u} e^{-\sum_{k}(\beta^k_r-\beta_s^k) \Sigma_2^{(k)}}
		\cdot 	\begin{pmatrix}
			\sum_k c_1^{(k)} \circ \d \beta^k_r \\ \vdots \\ \sum_k c_d^{(k)} \circ \d \beta^k_r	
		\end{pmatrix} \d u, \\
		V_t
		=&	e^{\sum_{k} (\beta^k_t-\beta^k_s) \Sigma_2^{(k)}} v \\
		&+	e^{\sum_{k} (\beta^k_t-\beta^k_s) \Sigma_2^{(k)}} 
		\cdot \int_{s}^{t} e^{-\sum_{k} (\beta^k_u-\beta^k_s) \Sigma_2^{(k)}}
		\cdot 	\begin{pmatrix}
			\sum_k c_1^{(k)} \circ \d \beta^k_u \\ \vdots \\ \sum_k c_d^{(k)} \circ \d \beta^k_u	
		\end{pmatrix}.
	\end{align*}
	To show global in time dispersion we have to calculate the determinant of 
	\begin{align*}
		D_v\Phi_{s,t}(x,v)^{(1)} = 	\int_{s}^{t} e^{\sum_{k} (\beta^k_u-\beta_s^k) \Sigma_2^{(k)}} \d u
	\end{align*}
	and show that 	\begin{align*}
		\left|\det D_v\Phi_{s,t}(x,v)^{(1)}\right| \geq |t-s|^d.
	\end{align*}
\eqref{it:glob_disp1} To understand the solution, consider first $\sigma^k \equiv \ const \ $ for all $k$. In this case, the matrix $\Sigma_2^{(k)}$ is equal to $0$ and thus, the first components of the Jacobian matrix are given by
		\begin{align*}
			D_v\Phi_{s,t}(x,v)^{(1)} = 	(t-s)E_d.
		\end{align*}
		Consequently, for the determinant, we obtain
		\begin{align*}
			|\det D_v\Phi_{s,t}(x,v)^{(1)}| = 	|t-s|^d.
		\end{align*}
		Now, let us consider the case where $\Sigma_2 \neq 0$ and $\Sigma_2$ is nilpotent. We calculate
		\begin{align*}
			e^{\sum_k (\beta^k_u-\beta_s^k) \Sigma_2^{(k)}}
			& = E_d + \sum_{n=1}^{d-1}\frac{\left(\sum_{k=1}^{N} (\beta^k_u-\beta_s^k)\Sigma_2^{(k)}\right)^n}{n!}.
		\end{align*}
		Thus, we rewrite
		\begin{align*}
			&\det D_v \Phi_{s,t}(x,v)^{(1)} \\
			&= \det \left( (t-s) E_d - D \right)
		\end{align*}
		where $D$ is defined as 
		\begin{align*}
D = \left(-\sum_{n=1}^{d-1}\frac{1}{n!}\int_s^t \left(\sum_{k=0}^{N} \left(\beta^k_u-\beta_s^k\right)\Sigma_2^{(k)}\right)^n \d u\right).
			\end{align*} 
		Since D is the sum of nilpotent matrices, D itself is nilpotent. Thus, $D$ has only $0$ as its eigenvalue. Consequently, the above expression, which is the characteristic polynomial of $D$ simplifies to 	
		\begin{align*}
		 \det \left( (t-s) E_d - D \right) =(t-s)^d. 
	\end{align*}
		We note, provided that $\Sigma_2$ is nilpotent, then $\Sigma_2^{(k)}$ is nilpotent for all $k$. The converse is not generally true, having $\Sigma_2^{(k)}$ is nilpotent for all $k$ does not necessarily imply that $\Sigma_2$ itself is nilpotent.\\ 
\eqref{it:glob_disp2} If we consider a combination of Brownian motions for $N >1$, we can combine matrices which are all diagonal. Consider the system
\begin{align*}
	d X(t) &= V(t) \d t\\
	d V(t) &= 
	\sum_{k}	\underbrace{\begin{pmatrix}
			\lambda^{(k)}_{1} 	& 		& 0 \\
			& \ddots& \\
			0			& 		& \lambda^{(k)}_{d}			
		\end{pmatrix}
	}_{=: \Sigma_2^{(k)}}V(t)\circ \d \beta^k_t 
\end{align*}
with $\sum_{i=1}^{d} \lambda^{(k)}_i = 0$ for all $k$. 
Then, using $d$-times Hölder's inequality, the determinant of the Jacobian matrix is given by
\begin{align*}
	|\det D_v \Phi_{s,t} (x,v)^{(1)}|
	&= 	\prod_{i=1}^{d} \left(	\int_{s}^{t}	\prod_{k}e^{\lambda^{(k)}_{i}(\beta^k_u-\beta^k_s)}\d u \right) \\
	&\ge 	\left(
	\int_{s}^{t} \prod_{i=1}^{d} \left(	\prod_{k}e^{\lambda^{(k)}_{i}(\beta^k_u-\beta^k_s)} \right)^{\frac{1}{d}} \d u
	\right)^d	\\
	&=
	\left(
	\int_{s}^{t} \underbrace{\left(	e^{\sum_{k}\sum_{i=1}^{d} \lambda_i^{(k)} \cdot (\beta^k_u-\beta^k_s)} \right)^{\frac{1}{d}}}_{=1} \d u
	\right)^d	\\
	&=	(t-s)^d.
\end{align*}\\		
\eqref{it:glob_disp3} Let us now consider matrices in Jordan normal form for $N = 1$. If the coefficients are given by a diagonal matrix, then the result is true due to \eqref{it:glob_disp2}. Assume that the matrix $\Sigma_2$ is in Jordan normal form with several Jordan blocks. Then, $\Sigma_2$ can be expressed as
			\begin{align*}
				\Sigma_2 = \begin{pmatrix}
					J_1 & 		& \\
					&\ddots &	\\
					&		& J_n
				\end{pmatrix}
				\text{with }
				J_i = 	\begin{pmatrix}
					\lambda_i 	& 1 	& 	 	& 	\\
					&\ddots & \ddots&	\\ 
					&		& \ddots& 1	\\
					&		&		& \lambda_i
				\end{pmatrix}
			\end{align*}
			with constraint
			\begin{align*}
				\sum_{i=1}^{n} \sum_{j_i=1}^{n_i} \lambda_i = \sum_{i=1}^{n} n_i \lambda_i = 0\text{ and }
				\sum_{i=1}^{n} n_i = d.
			\end{align*}
			Therefore, we rewrite the Jacobian as 
			\begin{align*}
				D_v\Phi_{s,t}(x,v)^{(1)} &= \int_{s}^{t} e^{ (\beta_u-\beta_s) \Sigma_2} \d u = \diag\left(	\int_s^t	e^{ (\beta_u-\beta_s)J_i}\d u\right)_{i=1,\dots n}.
			\end{align*}
Define $A(s,u) \defeq (a_{ij}(s,u))_{i,j \in \{1,\dots,n\}}$ with $a_{ij}(s,u)\defeq \begin{cases}
	(\beta_u-\beta_s)^{j-i}, &j \geq i\\
	0, & j < i
	\end{cases}.$ With this, the matrix exponential of the Jordan blocks above is given by
			\begin{align*}
				&e^{ (\beta_u-\beta_s)J_i}=	e^{(\beta_u-\beta_s)\lambda_i } \cdot A(s,u).		
			\end{align*}
		Consequently, using $d$-times Hölder's inequality again, the determinant of this matrix is given by 
			\begin{align*}
				\det (D_v\Phi_{s,t}(x,v)^{(1)})
				=	\prod_{i=1}^{n} \prod_{j_i=1}^{n_i} \left(\int_{s}^{t} e^{\lambda_i (\beta_u-\beta_s)} \d u\right) \ge			(t-s)^d.
			\end{align*}
Finally, assume that the matrix $\Sigma_2$ is given by $\Sigma_2 = SBS^{-1}$ with $B$ a Jordan matrix. We rewrite
			\begin{align*}
				&e^{(\beta_u-\beta_s)\Sigma_2}
				=	e^{(\beta_u-\beta_s)SBS^{-1}}
				=	Se^{(\beta_u-\beta_s)B } S^{-1}.
			\end{align*}
			Therefore, using the above shown result for a matrix in Jordan normal form with several Jordan blocks we obtain
			\begin{align*}
				\det \left(D_v\Phi_{s,t}(x,v)^{(1)}\right)
				&=	\det S \det \left(\int_s^t e^{(\beta_u-\beta_s)B }\d u\right) \det (S^{-1}) \\
				&= \det \left(\int_s^t e^{(\beta_u-\beta_s)B }\d u\right) \ge			(t-s)^d.
			\end{align*}
\end{proof}

There are also affine linear cases where Assumption \ref{ass:disp_glob} is not fulfilled.
\begin{example} 
	If d = 2, $\sigma_1 (x,v) = \sigma_1(v)$ is affine linear, $\sigma^k \equiv 0\quad \forall k \neq 1 $ and $\det \Sigma_2 > 0$ for 
	\begin{align*}
		\Sigma_2& = 	\begin{pmatrix}		
			\mu_{11}\circ d\beta^1_t & 	\mu_{12}\circ d\beta^1_t 	 \\
			\mu_{21}\circ d\beta^1_t 	& 	-\mu_{11}\circ d\beta^1_t 
		\end{pmatrix}
	\end{align*}
	then Assumption \ref{ass:disp_glob} cannot hold true for all $\tau \in (0,\infty).$
\end{example}

\begin{proof}
	Given a $2$-dimensional matrix $\Sigma_2$ with $\tr(\Sigma_2) = 0$ we can calculate the matrix $\Sigma_2$ to the power of $k$ explicitly. For $k \in \N$ we obtain
	\begin{align*}
		(\Sigma_2)^{2k}	&=	(-1)^k (\det \Sigma_2)^k \cdot E_2,	\\
		(\Sigma_2)^{2k+1}	&=	(-1)^k (\det \Sigma_2)^k \cdot \Sigma_2. 
	\end{align*} 
Using this expression, the matrix exponential $e^{\beta_t \Sigma_2}$ is given by
	\begin{align*}
		&e^{\beta_t \Sigma_2} 	\\
		= 	&\sum_{K=0}^{\infty} \frac{(-1)^k (\det \Sigma_2)^k}{(2k)!} \beta_t^{2k} E_2	
		+ 	 \sum_{K=0}^{\infty} \frac{(-1)^k (\det \Sigma_2)^k}{(2k+1)!} \beta_t^{2k+1} \Sigma_2	\\
		=	&\begin{cases}
			E_2 + \beta_t \Sigma_2, \hfill \det \Sigma_2 = 0\\
			\cos\left(\det(\Sigma_2)^{\frac{1}{2}} \beta_t\right) \cdot E_2 
			+ 	\sin\left(\det(\Sigma_2)^{\frac{1}{2}} \beta_t\right) \cdot \det\left(\Sigma_2\right)^{-\frac{1}{2}} \Sigma_2 , \hfill \det \Sigma_2 > 0	\\
			\cosh\left(\det(\Sigma_2)^{\frac{1}{2}} \beta_t\right) \cdot E_2 
			+ 	\sinh\left(\det(\Sigma_2)^{\frac{1}{2}} \beta_t\right) \cdot \left(-\det(\Sigma_2)^{-\frac{1}{2}}\right) \Sigma_2 , \hfill \det \Sigma_2 < 0	
		\end{cases}.
	\end{align*}
With this expression, the determinant of the Jacobian matrix of the first components of the stochastic flow is given by 
	\begin{align*}
		A:= D_v (\Phi_{0,t}^{-1})^{(1)} = 	\int_{0}^{t} e^{\beta_s \Sigma_2} \d s.
	\end{align*} 
	If $\det \Sigma_2 > 0$, the Matrix $A$ can be expressed as
		\begin{align*}
		A =C(t) \cdot E_2	+	 S(t) \cdot (\det \Sigma_2)^{-\frac{1}{2}}\cdot \Sigma_2	
	\end{align*}
	with $C(t) = \int_{0}^{t} \cos\left(\beta_s (\det \Sigma_2)^{\frac{1}{2}}\right) \d s$ and $S(t) = \int_{0}^{t} \sin\left(\beta_s (\det \Sigma_2)^{\frac{1}{2}}\right) \d s$. Its determinant is calculated with respect to $C(t)$ and $S(t)$ by 
	\begin{align*}
		\det A	&=	C(t)^2 + S(t)^2 (\det \Sigma_2)^{-1} \det \Sigma_2	=	C(t)^2 + S(t)^2.
	\end{align*}
	This is always positive. But we cannot expect that for almost all $\omega$ there exists $c>0$ such that $C(t)^2 + S(t)^2 \ge c \cdot t^2$ for all $t$ and thus Assumption \ref{ass:disp_glob} cannot hold true for all $\tau \in (0,\infty)$. Specifically, using addition theorems we have
	\begin{align*}
		\det D_v \Phi_{t}^{(1)} =
		&\left( \int_{0}^{t} \cos(\beta_s) \d s	\right)^2
		+	\left( \int_{0}^{t} \sin(\beta_s) \d s	\right)^2	\\
		=	&\int_{0}^{t}\int_{0}^{t} \cos(\beta_s) \cos(\beta_u) \d s \d u
		+	\int_{0}^{t}\int_{0}^{t} \sin(\beta_s) \sin(\beta_u) \d s \d u	\\
		=&	\int_{0}^{t}\int_{0}^{t} \cos(\beta_s-\beta_u) \d s \d u.
	\end{align*}
	In order to simplify notation we assume $\det \Sigma_2 = 1$. The general case can be calculated analogously. 
The Brownian motion $\beta_t$ is normal distributed with density $f(x) = \frac{1}{\sqrt{2\lambda t}} e^{-\frac{x^2}{2t}}$.
	Thus, given the matrix $\Sigma_{s,u} = \begin{pmatrix}s & \min(s,u) \\ \min(s,u) & u\end{pmatrix}$ the tuple $\begin{pmatrix} \beta_s & \beta_u \end{pmatrix}$ 
	is normal distributed with joint density
	\begin{align*}
		f_{s,u}(x,y) 
		= 	\frac{1}{2\lambda \sqrt{\det \Sigma_{s,u}}} 
		\exp\left( -\frac{1}{2} 
		\begin{pmatrix} x & y \end{pmatrix} 
		\Sigma_{s,u}^{-1} 
		\begin{pmatrix} x \\ y \end{pmatrix}
		\right).
	\end{align*}
	Thus, we calculate the expectation $\mathbb{E} \left(\int_{0}^{t}\int_{0}^{t} \cos(\beta_s-\beta_u) \d s \d u\right)$ explicitly by using these density and the fact that the integral of the density of the normal distribution is equal to $1$. With this we obtain
	\begin{align*}
		\mathbb{E} 	\left(
		\left(\int_{0}^{t} \cos(\beta_s) \d s\right)^2
		+\left(\int_{0}^{t} \sin(\beta_s) \d s\right)^2
		\right)
		= 4t-8+8\exp\left(-\frac{1}{2}t\right).
	\end{align*}
	Since 
	$\left(
	\left(\int_{0}^{t} \cos(\beta_s) \d s\right)^2
	+\left(\int_{0}^{t} \sin(\beta_s) \d s\right)^2
	\right)\ge 0\ \forall \omega \in \Omega$
	we know that there does not exist $c >0$ such that
	$\left(
	\left(\int_{0}^{t} \cos(\beta_s) \d s\right)^2
	+\left(\int_{0}^{t} \sin(\beta_s) \d s\right)^2
	\right)\ge c t^2$ for almost all $\omega \in \Omega$ and all $t$.
\end{proof}

\section{Stochastic versions of chemotactic movement}\label{sec:chemotaxis_det}
In this section, we will use the pathwise dispersion and Strichartz estimates shown above in order to show the main Theorem \ref{thm:main}, which yields the existence of weak martingale solutions to \eqref{IVP_stoch} starting from small initial data which is a stochastic analogue of \cite[Theorem 3]{BCGP2008}. To show the existence of a solution, we construct approximating solutions and prove a stability result. 

\subsection{Solution of a regularized chemotactic equation}
Let us first find a solution to a regularized stochastic chemotactic equation by assuming that the initial value and kernel fulfill additional regularity conditions. Note, that we neither assume further regularity conditions on the stochastic drift coefficients $\sigma_k$ nor on the Brownian motions $\beta^k$.\\
Note, that the 'tumbling' only takes place in a compact subset in the velocity variables due to the compact support of $K$. Nevertheless, the random movement allows for leaving  this set and therefore, solutions are defined on the whole space. Therefore, we distinguish between the set, where $f$ is effected by the presence of 'tumbling' and the set, where $f$ is a solution of the linear problem. Define $\bar{V} \defeq \{v \in \R^d: \exists s\leq t \in [0,T], x \in \R^d: \Psi_{s,t}(x,v)^{(2)} \in V \}$. $V$ is compact and the deviation due to the stochastic flow is bounded by a constant $C\left(\left\lceil\frac{\tilde{\tau}}{\tau}\right\rceil \right)$ due to \eqref{lem: assumption_local_disp:eq:differentiation} and \eqref{lem: assumption_local_disp:eq:hoelder_continuity}. Consequently, $\bar{V}$ is a compact domain, with size bounded by $\left\vert\bar{V}\right\vert \leq C(\left\vert V\right\vert, \left\lceil\frac{\tilde{\tau}}{\tau}\right\rceil)$. 

\begin{ass}\label{ass: regular}
Assume that there exist a parameter $a$ such that the kernel $K$ and initial value $f_0$ satisfy 
	\begin{enumerate}
		\item $f_0 \in L_{x,v}^a \cap L_{x,v}^1 $ is smooth, nonnegative, positive on  $\R^d \times \bar{{V}}$, bounded from above and supported in $\R^d \times \hat{V}$, where $\hat{V}\subseteq \R^d$ is compact with size $|\hat{V}| \leq C(|\bar{V}|)$.
		\item $K: L_x^1 \mapsto L_x^1L_v^1L_{v'}^1$ is Lipschitz continuous.
		\item $K: L_x^{\infty} \mapsto L_{x,v,v'}^{\infty}$ is Lipschitz continuous. 
		\item For all $p_1,r_1 \in [1,\infty]$ and $S \in L_t^{r_1}W_x^{1,p_1}$ the turning kernel $K(S)$ is smooth, bounded in $L^{\infty}(\R, \R^d, \R^d, \R^d)$ and compactly supported in $\R \times \R^d \times V \times V$ and satisfies Assumption \ref{ass:turning_kernel}.
	\end{enumerate}
\end{ass}

\begin{lem}\label{lem: approx_sol_measurable}
	Let $d \ge 2$. Fix $T \in (0, \infty)$ and consider parameters $r,a,p,q$ such that 
\begin{align*}
	r \in \left(2, \frac{d+3}{2}\right], 
	\quad
	r \ge a \ge 
	\max\left(\frac{d}{2},\frac{d}{d-1}\right) 
	\quad
	\frac{1}{p} = \frac{1}{a} - \frac{1}{rd}, \frac{1}{q} = \frac{1}{a} + \frac{1}{rd}.
\end{align*} Fix a stochastic basis $(\Omega, \mathcal{F}, \mathbb{P},(\mathcal{F}_t)_{t=0}^{T}, (\beta^k)_{k \in \N})$ and assume that the turning kernel $K$ and initial data $f_0$ fulfill Assumption \ref{ass: regular} with parameter $a$ as above and assume that the stochastic drift coefficients $\sigma^k$ fulfill Assumption \ref{ass: ex_flow}. Then, there exists an analytically weak, stochastically strong solution to \eqref{IVP_stoch}
	which has the following properties:
	\begin{enumerate}
		\item $f : \Omega \times [0, T] \rightarrow L^1_{x,v}$ is a nonnegative $(\mathcal{F}_t)_{t\in[0,T]}$ progressively measurable process.
		\item $f$ belongs to $L^2(\Omega,C_t (L^1_{x,v})) \cap L^{\infty}(\Omega \times [0, T] \times \R^{2d})\cap L^{\infty}(\Omega,L_t^r([0,T],L_x^pL_v^q))$.
	\end{enumerate}
\end{lem}

\begin{proof}
We start by constructing a sequence of approximations 
	$\{f^{k}\}_{k \in \N}$ over $[0,T]$ by 
	\begin{align*}
		f^{0} &= 0\\
		f^{k} &= f_0\circ \Psi_{0,t} + \int_{0}^{t} \int_V \left(K\left(S^{k}\right)\left(f^{k-1}\right)' - (K)^{\ast}\left(S^{k}\right)f^{k-1}\right)d v' \circ \Psi_{s,t} (x,v)\d s 
	\end{align*}
	with $S^{k} - \Delta S^{k} = \int f^{k-1} \d v$. First, since $|V|$ and $K$ are bounded,  $f^{k}$ is bounded by 	
	\begin{align*}
		\dorm{f^{k}}{\infty}{\infty}{\infty}\leq \norm{f_0}_{L_{x,v}^{\infty}}+  C T 	\dorm{f^{k-1}}{\infty}{\infty}{\infty},
	\end{align*}
	where $C$ is independent of $k$ but depends on the bound on $K$. 
If $T<C^{-1}$, using the iterative definition and geometric series we obtain
	\begin{align}\label{eq: bound_f_0}
		\dorm{f^{k}}{\infty}{\infty}{\infty}\leq (1-CT)^{-1}\norm{f_0}_{L_{x,v}^{\infty}}.
	\end{align}
	Let $X_T$ denote the Banach space of $(\mathcal{F}_t)_{t\in[0,T]}$ progressively measurable processes
	$f : \Omega \times [0, T] \rightarrow L^1_{x,v}$ endowed with the $L^2(\Omega,C_t (L^1_{x,v}))$-norm. Since the stochastic flow is $\mathbb{P}$-a.s. volume preserving, maximizing over $t$ yields $\mathbb{P}$-a.s.
	\begin{align*}
		&\dorm{f^{k+1}-f^{k}}{\infty}{1}{1} \\
		& \leq\norm{\int_0^t \zorm{ \int_V K\left(S^{k+1}\right)\left(\left(f^{k}\right)' - \left(f^{k-1}\right)'\right) \d v' }{1}{1}\d s}_{L_t^{\infty}}\\
		&+ \norm{\int_0^t \zorm{ \int_V \left(K \left(S^{k+1}\right)- K \left(S^{k}\right)\right)\left(f^{k-1}\right)'d v' }{1}{1}\d s}_{L_t^{\infty}} \\
		& + \norm{\int_0^t \zorm{\int_V \left(K \right)^{\ast}\left(S^{k+1}\right)\d v'\left(f^{k}- f^{k-1}\right) }{1}{1}\d s}_{L_t^{\infty}}\\
		&+ \norm{\int_0^t \zorm{\int_V \left((K )^{\ast}\left(S^{k+1}\right)- (K )^{\ast}\left(S^{k}\right)\right)\d v'f^{k-1} }{1}{1}\d s}_{L_t^{\infty}}\\
		& = I + II + III + IV.
	\end{align*}
	Using the boundedness of $K $, the iterative definition of $f^{k}$ and the geometric series, we obtain for some $C$ independent of $k$
	\begin{align*}
		I &\leq \norm{\int_0^t \norm{K \left(S^{k+1}\right)}_{L_x^{\infty}L_v^{1}L_{v'}^{\infty}}\zorm{f^{k}-f^{k-1}}{1}{1}\d s}_{L_t^{\infty}}\\
		&\leq C T\dorm{f^{k}-f^{k-1}}{\infty}{1}{1}.
	\end{align*}
	Using in addition the Lipschitz continuity of $K$, the fact that the solution $S^{k}$ is given as a convolution of the Bessel-Potential with the density $\rho_{k-1}$, and the integrability of the Bessel-Potential in $L^1$ we get
	\begin{align*}
		II \leq& \norm{\int_0^t \norm{K\left(S^{k+1}\right)-K\left(S^{k}\right)}_{L_x^{1}L_v^{1}L_{v'}^{1}}\zorm{f^{k-1}}{\infty}{\infty}\d s}_{L_t^{\infty}}\\
		\leq & C T\norm{ \norm{S^{k+1}-S^{k}}_{L_x^1}}_{L_t^{\infty}}\dorm{f^{k-1}}{\infty}{\infty}{\infty}\\
		\leq & C T \dorm{f^{k}-f^{k-1}}{\infty}{1}{1}\dorm{f^{k-1}}{\infty}{\infty}{\infty}\\
		\leq &\sum_{m=1}^{k-1}\left(CT\right)^m\norm{f_0}_{L_{x,v}^{\infty}}
		\dorm{f^{k}-f^{k-1}}{\infty}{1}{1}\\
		\leq &CT\left(1-CT\right)^{-1}\norm{f_0}_{L_{x,v}^{\infty}}\dorm{f^{k}-f^{k-1}}{\infty}{1}{1}.
	\end{align*}
	To calculate the third term we proceed as in the first one and obtain
	\begin{align*}
		III &\leq \norm{\int_0^t \norm{K(S^{k+1})}_{L_x^{\infty}L_v^{\infty}L_{v'}^{1}}\zorm{f^{k}-f^{k-1}}{1}{1}\d s}_{L_t^{\infty}}\\
		&\leq C T \dorm{f^{k}-f^{k-1}}{\infty}{1}{1}.
	\end{align*}
	For computing the fourth term we proceed as in the calculation of the second one and get
	\begin{align*}
		IV &\leq \norm{\int_0^t \norm{K\left(S^{k+1}\right)-K\left(S^{k}\right)}_{L_x^{1}L_v^{1}L_{v'}^{1}}\zorm{f^{k-1}}{\infty}{\infty}\d s}_{L_t^{\infty}}\\
		&\leq CT\left(1-CT\right)^{-1}\norm{f_0}_{L_{x,v}^{\infty}}\dorm{f^{k}-f^{k-1}}{\infty}{1}{1}.
	\end{align*}
	Combining these inequalities we obtain
	\begin{align}\label{eq: sup}
		\norm{f^{k+1}-f^{k}}_{X_T} 
		\leq CT\left(\left(1-CT\right)^{-1}\norm{f_0}_{L_{x,v}^{\infty}}+1\right)\norm{f^{k}-f^{k-1}}_{X_T}.
	\end{align}
	Therefore, for $T$ sufficiently small by Banach's fixpoint theorem, there exists a fixpoint $f$ in $X_T$. Applying this argument a finite number of times allows us to remove the constraint on $T$. More precisely, let $i \in \N$ be an integer and $T_i = \frac{1}{C\cdot(i+1)} \frac{1}{\norm{f_0}_{L_{x,v}^{\infty}}+1}$. Thus, equation \eqref{eq: bound_f_0} gives
	\begin{align*}
		\norm{f_{\sum_{j=1}^{i}T_j}}_{L_{x,v}^{\infty}}\leq (1-CT_{i})^{-1}\norm{f_{\sum_{j=1}^{i-1}T_j}}_{L_{x,v}^{\infty}}<\frac{i+1}{i}\norm{f_{\sum_{j=1}^{i-1}T_j}}_{L_{x,v}^{\infty}}<(i+1)\norm{f_0}_{L_{x,v}^{\infty}}.
	\end{align*}
	Consequently, the prefactor in equation \eqref{eq: sup} gives 
	\begin{align*}
		 CT_i\left(\left(1-CT_i\right)^{-1}\norm{f_{\sum_{j=1}^{i-1}T_j}}_{L_{x,v}^{\infty}}+1\right) &< CT_i \left((i+1)\norm{f_0}_{L_{x,v}^{\infty}}+1\right)\\&< CT_i(i+1) \left(\norm{f_0}_{L_{x,v}^{\infty}}+1\right)=1.
	\end{align*} 
	This ensures that the conditions for using Banach's fixpoint theorem are satisfied. Furthermore, calculating
	\begin{align*}
	\sum_{i=1}^m T_i=\frac{1}{C(\norm{f_0}_{L_{x,v}^{\infty}}+1)}\sum_{i=1}^m \frac{1}{(i+1)}, 
	\end{align*}
	we obtain a divergent sequence. Thus, applying the above argument a finite number of times allows us to remove the constraint on $T$.
	 Consequently, there exists $f \in X_T$ such that ${f^{k}}$ converges to $f$ in $L^2(\Omega,C_t (L^1_{x,v}))$. \\
	Secondly, we aim to show that $f$ belongs to $L^{\infty}(\Omega \times [0, T] \times \R^{2d})$.
	Repeating the argument in equation \eqref{eq: bound_f_0} a finite number of times we can remove the restriction on $T$. Taking $L^{\infty}(\Omega)$-norms on both sides of the above inequality yields
	the uniform bound. By weak-* $L^{\infty}$ sequential compactness of balls in $L^{\infty}(\Omega \times [0, T] \times \R^{2d})$, each $f$ is an element of $L^{\infty}(\Omega \times [0, T] \times \R^{2d})$. Moreover, repeating the above calcution with $X_T$ replaced by $L^{\infty}(\Omega \times [0, T] \times \R^{2d})$, we deduce that $f$ is also the fixpoint solution of $f^k$ in $L^{\infty}(\Omega \times [0, T] \times \R^{2d})$. More precisely, due to the boundedness and Lipschitz continuity of $K$ both in $L_x^1$ and $L_x^{\infty}$ we can follow the same calculation as above.  
	
	We next demonstrate that $f$ is $\mathbb{P}-a.s.$ nonnegative. 
	Define the stopping-time $$t^{\ast}(\omega) \defeq \inf \{t \in [0,T]: \exists v \in V, \exists x \in X: \inf(f(t,x,v))<0 \}.$$ Since $f_0$ and $K$ are smooth and $K$ is compactly supported, $f^k$ is continuous for each $k$. Since $f^k$ converges to $f$ in $L^{\infty}([0,T]\times \R^{2d})$ the fixpoint solution $f$ is continuous. Recall that the turning kernel $K(S)$ is supported in the compact set $V$ in the velocity variables and in the compact set $X$ in the $x$ variables. Moreover, since $f_0(x,v)$ is positive, and, $|V|$ and $|X|$ are bounded, we deduce $t^{\ast}(\omega)>0$ for almost all $\omega$. Assume $t^{\ast}(\omega) < T$. By using the relation $\Psi_{s,t} \circ \Phi_{0,t} = \Phi_{0,s}$ and integration by parts, we obtain $\mathbb{P}$-almost surely that 
	\begin{align*}
		&e^{\int_0^{t^{\ast}}\int_V K\left(S\right)(s,\Phi_{0,s}(x,v)^{(1)} ,v',\Phi_{0,s} (x,v)^{(2)}) \d v'\d s} f({t^{\ast}},\Phi_{0,{t^{\ast}}}(x,v))\\
		&= f_0(x,v) + \int_0^{t^{\ast}} e^{\int_0^s\int_V K\left(S\right)(s,\Phi_{0,u}(x,v)^{(1)},v',\Phi_{0,u} (x,v)^{(2)}) \d v' \d u}\\
		&\cdot \int_{V}K\left(S\right)(s,\Phi_{0,s}(x,v),v')\left(f(s,\Phi_{0,s}(x,v)^{(1)},v')\right)\d v' \d s\\
		&\geq f_0(x,v).
	\end{align*}
Replacing $(x,v)$ by $\Psi_{0,{t^{\ast}}}(x,v)$ gives for all $x \in X$ and all $v \in V$ 		\begin{align*}
	f({t^{\ast}},x,v) &=  f({t^{\ast}}, \Phi_{0,{t^{\ast}}}(\Psi_{0,{t^{\ast}}}(x,v)))\\
	&\geq e^{-\int_0^{t^{\ast}}\int_V (K)^{\ast}\left(S\right) \d v' \circ \Phi_{0,s}(\Psi_{0,{t^{\ast}}}(x,v)) \d s}f_0(\Psi_{0,{t^{\ast}}}(x,v)).
\end{align*}
Note, that for $v \in V$ the velocity variable $\Psi_{0,{t^{\ast}}}(x,v)^{(2)}$ is an element of $\bar{V}$, the set where $f$ is effected by the tumbling and thus,  $f_0(\Psi_{0,{t^{\ast}}}(x,v))>0$.
	Hence,  for all $x \in X$ and all $v \in V$ the function  $f({t^{\ast}},x,v)$ is $\mathbb{P}$-almost surely positive and, since $f$ is continuous, for almost all $\omega$ there is $\tilde{t}(\omega)>t^{\ast}(\omega)$ such that $f({\tilde{t}},x,v)>0$ for all  $x \in X$ and all $v \in V$, in contradiction to the definition of $t^{\ast}$. Hence, $t^{\ast} = T$. Hence, for all  $(x,v) \in \R^{2d}$ and all $t \in [0,T]$ we deduce $\mathbb{P}$-almost surely that
				\begin{align*}
				f(t,x,v) &= f_0\circ \Psi_{0,t}(x,v) + \int_{0}^{t} \int_V \left(K\left(S\right)\left(f\right)' - (K)^{\ast}\left(S\right)f\right)d v' \circ \Psi_{s,t} (x,v)\d s \\
				 &= f_0\circ \Psi_{0,t}(x,v) \geq 0.
			\end{align*}
	Finally, we aim to show that $f \in L^{\infty}(\Omega,L_t^r([0,T],L_x^pL_v^q))$. Define the parameters $(\tilde{r}, \tilde{p}, \tilde{q})$ by 
	\begin{align*}
		&\frac{1}{\tilde{q}} = \frac{1}{a'} + \frac{1}{d} - \frac{2}{rd}, \quad \quad \frac{1}{\tilde{p}} = \frac{1}{a'} - \frac{1}{d} + \frac{2}{rd}, \quad \quad \frac{1}{\tilde{r}} = 1 - \frac{2}{r}. 
	\end{align*}
	Then, $(a,r,p,q)$ and $(a',\tilde{r}, \tilde{p}, \tilde{q})$ are jointly admissible tuples (comp. Definition \ref{def: admissible}). Using the Strichartz estimates Lemma \ref{lm: stoch_strich} we obtain 	
	\begin{align*}
		&\dorm{f^{k}}{r}{q}{p}\\
		&\leq C\left(\left\vert V\right\vert, \left\lceil\frac{T}{\tau}\right\rceil\right)\zorm{f_0}{a}{a} + C\left(\left\vert V\right\vert, \left\lceil\frac{T}{\tau}\right\rceil\right)\dorm{ \int_V K\left(S^{k}\right)\left(f^{k-1}\right)'dv'}{\tilde{r}'}{\tilde{p}'}{\tilde{q}'}.
	\end{align*}
	With a similar calculation as for the a-priori-estimates Lemma \ref{lem:a-priori}, and with the boundedness and compact support of $K$ we estimate 
	\begin{align*}
		\dorm{ \int_V K\left(S^{k}\right)\left(f^{k-1}\right)'dv'}{\tilde{r}'}{\tilde{p}'}{\tilde{q}'} 	&\leq\norm{\norm{K\left(S^{k}\right)}_{L_x^{\alpha}L_v^{\tilde{q}'}L_{v'}^{q'}}\zorm{f^{k-1}}{p}{q}}_{L_t^{\tilde{r}'}[0,T]} \\
		&\leq \norm{ C \zorm{f^{k-1}}{p}{q}}_{L_t^{\tilde{r}'}[0,T]} 	\\
		&\leq C T^{\frac{1}{r}}\dorm{f^{k-1}}{r}{p}{q}.
	\end{align*} 
	Thus, with the iterative definition and geometric summation for $T(\omega)$ small enough we obtain
	\begin{align*}
		\dorm{f^{k}}{r}{q}{p}
		\leq & C\left(\left\vert V\right\vert, \left\lceil\frac{T}{\tau}\right\rceil\right)\sum_{m=0}^{k}\left(C\left(\left\vert V\right\vert, \left\lceil\frac{T}{\tau}\right\rceil\right)T^{\frac{1}{r}}\right)^m\norm{f_0}_{L_{x,v}^a}\\
		\leq & C\left(\left\vert V\right\vert, \left\lceil\frac{T}{\tau}\right\rceil\right)\left(1-C\left(\left\vert V\right\vert, \left\lceil\frac{T}{\tau}\right\rceil\right)T^{\frac{1}{r}}\right)^{-1}\norm{f_0}_{L_{x,v}^a}.
	\end{align*}
	Thus, $f^{k} \in L_t^rL_x^pL_v^q$ is uniform bounded for all $k$. Applying this argument a finite number of times, we can remove the constraint on $T$. By weak-* compactness we deduce that
	$f \in L^{\infty}(\Omega,L_t^r([0,T],L_x^pL_v^q))$. 
\end{proof}

\begin{remark}\label{rem:weak martingale solution}
	Since the solutions constructed above are stochastically strong solutions they are also stochastically weak (martingale) solutions and mild solutions. 
\end{remark}

\subsection{A-priori-estimates and bootstrapping}
The above solutions of a regularized kinetic model of chemotaxis satisfy the pathwise mild formulation, more precisely, $$f(\omega, t, x, v)=	f_0(\Phi_{0,t}^{-1}(x,v))  +\int_{0}^{t} \int_{V} \left(K^{\ast}(S)f - K(S)(f)'\d v'\right) \circ \Phi_{s,t}^{-1}(x,v) \d s .$$ 
First, we observe that the $L_{x,v}^1$-norm of any nonnegative solution $f$ to \eqref{IVP_stoch} is uniformly bounded in time for almost all $\omega$ thanks to conservation of mass provided that the $L_{x,v}^1$-norms of the initial data are uniformly bounded.

\begin{lem}\label{lem:conservation of mass}
	Assume that Assumption \ref{ass: regular} is fulfilled. Let $f$ be a nonnegative solution to the regularized chemotactic equation \eqref{IVP_stoch} and let $t \leq T$. Then, for almost all $\omega$ for the $L^1$-norm of $f$ we obtain 
	\begin{align*}
		\int_{\R^d}\int_{\R^d} f(\omega,t,x,v) \d v \d x = \int_{\R^d}\int_{\R^d} f_0( x,v) \d v \d x. 
	\end{align*}
\end{lem}

\begin{proof} We compute, using 
		\begin{align*}
		&\int_{0}^{t} \int_{V}\int_{V} \int_{\R^{d}} \left(K^{\ast}(S)f - K(S)(f^{n})'\d v'\right) \d s \d v \d x \\
		=	&\int_{0}^{t} \int_{\R^{d}} \int_{V}\int_{V} K(S)(s,x,v',v)f(s,x, v) \d v \d v' \d x \d s\\
		&-\int_{0}^{t} \int_{\R^{d}} \int_{V}\int_{V} K(S)(s,x,v,v')f(s,x, v') \d v' \d v \d x \d s =0,
	\end{align*}
	the support of the turning kernel in $V$, the mild formulation and the volume preservation of the stochastic flow to get
	\begin{align*}
		&\int_{\R^{d}} \int_{\R^d} f(\omega, t, x, v) \d v \d x \\
		=	&\int_{\R^{d}} \int_{\R^d} f_0(\Phi_{0,t}^{-1}(x,v)) \d v \d x \\
		&+\int_{\R^{d}} \int_{\R^d}\int_{0}^{t} \int_{V} \left(K^{\ast}(S)f - K(S)(f)'\d v'\right) \circ \Phi_{s,t}^{-1}(x,v) \d s \d v \d x\\
		=	&\int_{\R^{d}} \int_{\R^d} f_0(x,v) \d v \d x.
	\end{align*}

\end{proof}
Second, we show a-priori estimates in mixed $L^p$-spaces, which is the main step in showing the existence of solutions.
\begin{lem}\label{lem:a-priori}
	Let $d \ge 2$. Fix $T \in (0, \infty)$ and consider parameters $r,a,p,q$ such that 
	\begin{align*}
		r \in \left(2, \frac{d+3}{2}\right], 
		\quad
			r \ge a \ge 
			\max\left(\frac{d}{2}, \frac{d}{d-1}\right)
		\quad
		\frac{1}{p} = \frac{1}{a} - \frac{1}{rd}, \quad \frac{1}{q} = \frac{1}{a} + \frac{1}{rd}.
	\end{align*} 
	Assume that the turning kernel $K$ is supported in $\R \times \R^d \times V \times V$ and satisfies Assumption \ref{ass:turning_kernel} with a constant $\tilde{C}$ and either Assumption \ref{ass:disp_local} or Assumption \ref{ass:disp_glob} is valid for some $\tau$. Assume additionally, that $K$ and $f_0$ fulfill Assumption \ref{ass: regular}. Then, for any $\tilde{\tau}$ in the interval $\tau \leq \tilde{\tau} \leq T$ there exists $C\left(\left\lceil\frac{\tilde{\tau}}{\tau}\right\rceil \right) > 0 $ depending only on the quotient $\frac{\tilde{\tau}}{\tau}$, $\tilde{C}$ and further deterministic parameters, such that for all nonnegative solutions $f$ to \eqref{IVP_stoch}, we have $\mathbb{P}$-almost surely
	\begin{align}\label{lem: a-priori_eq_1}
		\Vert f \Vert_{L_t^r([0,\tilde{\tau}] L_x^p L_v^q)}
		&\le C\left(\left\lceil\frac{\tilde{\tau}}{\tau}\right\rceil \right)\Vert f_0 \Vert_{L^a(\R^{2d})} 
		+ C\left(\left\lceil\frac{\tilde{\tau}}{\tau}\right\rceil \right) \Vert f \Vert_{L_t^r([0,\tilde{\tau}] ,L_x^p L_v^q)}^2.
	\end{align}
	If there exists a $m \in \N$ such that 
$\Vert f_0\Vert_{L^a(\R^{2d})}< C^{-2}(m)\frac{1}{8}$, there is a stopping-time $\tau^{\ast}\defeq \min(m \tau,T)$ and $C$ independent of $\omega$ and $f$ such that $\mathbb{P}$-almost surely
	\begin{align}\label{lem: a-priori_eq_2}
		\Vert f \Vert_{L_t^r([0,\tau^{\ast}], L_x^p L_v^q)}\le C.
	\end{align}
	For a sequence of initial values $f_0^k \in L^1(\R^{2d}) \cap L^{a}(\R^{2d})$ with $\norm{f_0^k}_{L^{a}(\R^{2d})}$ converging to zero, the maximal existence time $\tau^{\ast}$ converges to $T$.
\end{lem}

\begin{proof}
	First, $f_0$ and $K$ are supported in a bounded domain in the velocity variable, and the deviation due to the stochastic flow is bounded by a constant $C\left(\left\lceil\frac{\tilde{\tau}}{\tau}\right\rceil \right)$ due to \eqref{lem: assumption_local_disp:eq:differentiation} and \eqref{lem: assumption_local_disp:eq:hoelder_continuity}. Consequently, the solution $f$ will be supported in a bounded domain $\tilde{V}$, with size bounded by $\left\vert\tilde{V}\right\vert \leq C(\left\vert V\right\vert, \left\lceil\frac{\tilde{\tau}}{\tau}\right\rceil)$. Since $K(S)$ and $f$ are always nonnegative, we obtain 
\begin{align*}
f(\omega,t,x,v)
&=	f_0 \circ \Phi_{0,t}^{-1} 
+	\int_{0}^{t}\int_{V}(K)^{\ast}(S)f - K(S)(f)' \d v' \circ \Phi_{s,t}^{-1} \d s \\
& \leq f_0 \circ \Phi_{0,t}^{-1} 
+	\int_{0}^{t}\int_{V}(K)^{\ast}(S)f \d v' \circ \Phi_{s,t}^{-1} \d s.
\end{align*}
Define the parameters $(\tilde{r}, \tilde{p}, \tilde{q})$ by 
	\begin{align*}
		&\frac{1}{\tilde{q}} = \frac{1}{a'} + \frac{1}{d} - \frac{2}{rd}, \quad \quad \frac{1}{\tilde{p}} = \frac{1}{a'} - \frac{1}{d} + \frac{2}{rd}, \quad \quad \frac{1}{\tilde{r}} = 1 - \frac{2}{r}. 
	\end{align*}
	Then, $(a,r,p,q)$ and $(a',\tilde{r}, \tilde{p}, \tilde{q})$ are jointly admissible tuples (comp. Definition \ref{def: admissible}).
	Thus, by applying the Strichartz estimates Lemma \ref{lm: stoch_strich}, we get for almost all $\omega$
	\begin{align}\label{lem_a-priori: eq_strich}
		\Vert f\Vert_{L_t^r L_x^p L_v^q}
		\le C\left(\frac{\tilde{ \tau}}{\tau}\right) \Vert f_0\Vert_{L_{x,v}^a} + C\left(\frac{\tilde{ \tau}}{\tau}\right)\left\Vert \int_{V} K(S) (f)' \d v'\right\Vert_{L_t^{\tilde{r}'} L_x^{\tilde{p}'} L_v^{\tilde{q}'}}.
	\end{align}
	To estimate the second term we apply Hölder's inequality and obtain
	\begin{align*}
		\int_{V} K(S)(t,x,v,v') f(t,x,v') \d v'
		\le \Vert K(S)(t,x,v,v')\Vert_{L_{v'}^{q'}} \cdot \Vert f(t,x,v')\Vert_{L_{v'}^{q}}.
	\end{align*}
	Thus, 
	\begin{align*}
		\left\Vert \int_{V} K(S) (f)' \d v'\right\Vert_{L_v^{\tilde{q}'}}
		\leq \Vert K(S)(t,x,v,v')\Vert_{L_v^{\tilde{q}'} L_{v'}^{q'}} \cdot \Vert f(t,x,v')\Vert_{L_{v'}^{q}}.		
	\end{align*}
	By Hölder's inequality and considering the relation 
	\begin{align*}
		\frac{1}{\tilde{p}'} = \frac{1}{p} + \frac{1}{\alpha},
	\end{align*}
	which leads to
	\begin{align*}
			\frac{1}{\alpha}
		= \frac{1}{d} - \frac{1}{rd}
	\end{align*}
	we obtain
	\begin{align*}
		\left\Vert \int_{V} K(S)(f)' \d v' \right\Vert_{L_x^{\tilde{p}'} L_v^{\tilde{q}'}}
		\le \Vert K(S)(t,x,v,v')\Vert_{L_x^\alpha L_v^{\tilde{q}'} L_{v'}^{q'}} \cdot \Vert f(t,x,v')\Vert_{L_x^p L_{v'}^q}.
	\end{align*}
	Since the parameters satisfy $r \le \frac{d+3}{2}$ and $a\geq \frac{d}{2} \geq \frac{d}{d-1}$ we estimate
	\begin{align*}
		\frac{1}{\alpha} = 	\frac{1}{d} - \frac{1}{rd} 
		\le	\frac{1}{\tilde{q}'} - \frac{d-2r+3}{dr}
		\le	\frac{1}{\tilde{q}'},
	\end{align*}
and
\begin{align*}
			\frac{1}{\alpha} 
	= 	\frac{1}{d} - \frac{1}{rd} 
		=	\frac{1}{q'} + \frac{d-ad+a}{ad}
		\le 	\frac{1}{q'}. 
	\end{align*}
 Therefore, with Assumption \ref{ass:turning_kernel} we obtain
	\begin{align*}
		\Vert K(S) \Vert_{L_{x}^{\alpha} L_{v}^{\tilde{q}'} L_{v}^{q'}}	
		&\le C(|V|, \tilde{q}', q') \left(\Vert S(t,x) \Vert_{L_x^\alpha} + \Vert \nabla S(t,x)\Vert_{L_x^\alpha}\right) \\
		&=	C(|V|, \tilde{q}', q') \left(\Vert G \ast \rho(t)\Vert_{L_x^\alpha} + \Vert \nabla G \ast \rho(t)\Vert_{L_x^\alpha}\right), 
	\end{align*}
	where $G$ is the Bessel potential. According to Aronszajn and Smith (\cite[ eq. 4.2 and 4.3 with $\alpha =2$]{AS1961}) the Bessel potential $G$ satisfies $G \in L^b$ for any $b < \frac{d}{d-2}$ and $\nabla G \in L^b$ for any $b < \frac{d}{d-1}$. 
	Define $b$ by $1 + \frac{1}{\alpha} = \frac{1}{b} + \frac{1}{p}$. The relation $\frac{1}{a} \le \frac{2}{d}$ implies
	\begin{align*}
		\frac{1}{b} 
		= 	1 + \frac{1}{\alpha} - \frac{1}{p} 
		=	1 + \frac{1}{d} - \frac{1}{rd} - \frac{1}{a} + \frac{1}{rd}
		\ge \frac{d + 1 -2}{d}
		=	\frac{d-1}{d}.
	\end{align*}
 Thus, the parameter $b$ is bounded by $b \le \frac{d}{d-1}$. This implies that $G \in L^b$ for all $b \le \frac{d}{d-1}$. To estimate $\Vert G \ast \rho(t)\Vert_{L_x^\alpha}$ we proceed by using Young's inequality
 	\begin{align*}
 	\Vert G \ast \rho(t) \Vert_{L_x^\alpha} 
 	\le \Vert G\Vert_{L^b} \cdot \Vert \rho\Vert_{L_x^p}	\le C(b) \cdot \Vert \rho(t,x)\Vert_{L_x^p}	
 	\overset{q \ge 1}{\le} 
 	C\left(b, q, |V|,\left\lceil\frac{\tilde{\tau}}{\tau}\right\rceil\right) \cdot \Vert f(t,x,v) \Vert_{L_x^p L_v^q}.
 \end{align*} 
Similarly, for $b<\frac{d}{d-1}$ we use Young's inequality
		\begin{align*}
		\Vert \nabla G \ast \rho(t) \Vert_{L_x^\alpha} \le 
		C\left(b, q, |V|,\left\lceil\frac{\tilde{\tau}}{\tau}\right\rceil\right) \cdot \Vert f(t,x,v) \Vert_{L_x^p L_v^q}.
	\end{align*} 
	If $b = \frac{d}{d-1}$, we estimate $\Vert \nabla G \ast \rho(t)\Vert_{L_x^\alpha}$ by using the
	Hardy-Littlewood-Sobolev inequality. Since $\Vert \nabla G(x)\Vert \le \frac{C}{|x|^{d-1}}$ for all $x$ and $\frac{1}{p} - \frac{1}{d} = \frac{1}{\alpha}$ with \cite[S. 119, Theorem 1]{S1970} we obtain
	\begin{align*}
		\Vert \nabla G \ast \rho\Vert_{L_x^\alpha}
		\le	\left\Vert \frac{C}{|x|^{d-1}} \ast \rho\right\Vert_{L_x^\alpha}
\le 
		C \Vert \rho \Vert_{L_x^p}
		\overset{q\ge 1}{\le} 
		C\left(q, |V|,\left\lceil\frac{\tilde{\tau}}{\tau}\right\rceil\right) \cdot \Vert f \Vert_{L_x^p L_v^q}. 
	\end{align*}
	Altogether, this shows
	\begin{align*}
		\left\Vert \int_{V} K(S) (f)' \d v' \right\Vert_{L_x^{\tilde{p}'} L_v^{\tilde{q}'}}
		\le C\left(a, |V|,\left\lceil\frac{\tilde{\tau}}{\tau}\right\rceil\right) \cdot \Vert f(t,x,v)\Vert_{L_x^p L_v^q}^2.
	\end{align*}
	Therefore, omitting the dependence on the parameters, we obtain
	\begin{align*}
		\left\Vert \int_{V} K(S) (f)' \d v' \right\Vert_{L_t^{\tilde{r}'} L_x^{\tilde{p}'} L_v^{\tilde{q}'}}
		&\le C\left(\left\lceil\frac{\tilde{\tau}}{\tau}\right\rceil\right)\left\Vert \Vert f(t,x,v)\Vert_{L_x^p L_v^q}^2 \right\Vert_{L_t^{\tilde{r}'}}\\
		&\overset{\tilde{r}' = \frac{r}{2}}{=}
		C\left(\left\lceil\frac{\tilde{\tau}}{\tau}\right\rceil\right)\Vert f(t,x,v)\Vert_{L_t^r L_x^p L_v^q}^2.
	\end{align*}
	Combining this with equation \eqref{lem_a-priori: eq_strich} we have shown \eqref{lem: a-priori_eq_1}.
To address the second part, we will employ a bootstrap argument for small initial data. 
Let $\tau^*\defeq \min(m \tau,T)$. Assume that there exists an $\varepsilon \in (0,1)$ such that $\Vert f_0 \Vert_{L_{x,v}^a} \le C\left( m \right)^{-2} \frac{\varepsilon}{8}$.
	Define $T(\varepsilon)$ to be the maximal existence time where the norm of $f$ remains sufficiently small
	\begin{align*}
		T(\varepsilon) := \sup\left\{ \tilde{T} \in [0,\tau^{\ast}]: \Vert f \Vert_{L_t^r((0,\tilde{T}), L_x^p L_v^q)} \le \frac{1}{2} \varepsilon C\left(\left\lceil\frac{\tilde{\tau}}{\tau}\right\rceil \right)^{-1}\right\}.
	\end{align*} 
 Since the $L_t^r([0,\tilde{T}], L_x^p L_v^q)$-norm of $f$ is continuous with respect to $\tilde{T}$, the maximal existence time $T(\varepsilon)$ is guaranteed to be positive. To establish that $T(\varepsilon)$ is indeed maximal, we will proceed by contradiction. 
	Assume that $T(\varepsilon)<\tau^{\ast}$. Due to \eqref{lem: a-priori_eq_1} and $\varepsilon^2 \leq \varepsilon$ we have
	\begin{align*}
		\Vert f \Vert_{L_t^r((0,T(\varepsilon)), L_x^p L_v^q)}
		&\le		C\left(\left\lceil\frac{\tilde{\tau}}{\tau}\right\rceil \right) \Vert f_0\Vert_{L_{x,v}^a}
		+ 	C\left(\left\lceil\frac{\tilde{\tau}}{\tau}\right\rceil \right) \Vert f \Vert_{L_t^r((0,T(\varepsilon)), L_x^p L_v^q)}^2\\
		&\le\frac{3\varepsilon}{8} C\left(\left\lceil\frac{\tilde{\tau}}{\tau}\right\rceil \right)^{-1} < \frac{1\varepsilon}{2} C\left(\left\lceil\frac{\tilde{\tau}}{\tau}\right\rceil \right)^{-1}. 
	\end{align*}
	Since the $L_t^r([0, \tilde{T}], L_x^p L_v^q)$-norm is continuous with respect to $\tilde{T}$ there is a $T^{\ast} > T(\varepsilon)$ such that
	$\Vert f \Vert_{L_t^r((0,T^{\ast}), L_x^p L_v^q)} \le \frac{1}{2} \varepsilon C\left(\left\lceil\frac{\tilde{\tau}}{\tau}\right\rceil \right)^{-1}$. 
	Thus, $T(\varepsilon)$ is indeed maximal, $T(\varepsilon) =\tau^{\ast}$, and there is $C$ independent of $\omega$ such that $\Vert f \Vert_{L_t^r([0,\tau^{\ast}], L_x^p L_v^q)} \le C$. \\
	Given $f_0^k$ with $\Vert f_0^k \Vert_{L_{x,v}^a} \le C(1)^{-2} \frac{1}{8}$ we find the maximal existence time $\tau^{\ast,k}$ with 
	\begin{align*}
		\tau^{\ast,k} = \min\left\{\max \left\{m \in \N: \Vert f_0 \Vert_{L_{x,v}^a} < C\left(m \right)^{-2} \frac{1}{8}\right\}\cdot\tau,T\right\}.
	\end{align*}
	It remains to show that $\tau^{\ast,k}$ converges to $T$ if the norms of the initial values $\norm{f_0^k}_{L^{a}(\R^{2d})}$ converge to zero.  
	For every $\omega$ there exists an integer $k(\omega) \in \N$ such that $T \leq k(\omega) \tau$. Therefore, we obtain that $\tau^{\ast,k(\omega)} = T$ which gives pointwise convergence. With $\delta$ specified in Remark \ref{rem:large_dev} and provided that $\zorm{f_0}{a}{a}$ is sufficiently small (where the smallness depends only on $\tau$ and is independent of $\omega$) we assert that	$\mathbb{P}\left(\tau^{\ast} = T \right) \geq 1-\delta$.
\end{proof}

Next, we show a regularity estimate with respect to time.
\begin{lem}\label{lem: hoelder_regularity_t}
		Let $d \ge 2$. Fix $T \in (0, \infty)$ and consider parameters $r,a,p,q$ such that 
	\begin{align*}
		r \in \left(2, \frac{d+3}{2}\right], 
		\quad
		r \ge a \ge 
		\max\left(\frac{d}{2}, \frac{d}{d-1}\right)
		\quad
		\frac{1}{p} = \frac{1}{a} - \frac{1}{rd}, \frac{1}{q} = \frac{1}{a} + \frac{1}{rd}.
	\end{align*} 
Moreover, consider parameters $\lambda$ and $\kappa$ with $\lambda >\frac{4r}{r-2}$ and 
$\kappa \lambda + 1 = \lambda\left(\frac{1}{2}-\frac{1}{r}\right)$.
Assume additionally, that $K$ and $f_0$ fulfill Assumption \ref{ass: regular}.  Then, for any nonnegative solution $f$ to \eqref{IVP_stoch} and all $\varphi \in C_c^{\infty}\left(\R^{2d}\right)$ we have $\mathbb{P}$-almost surely
\begin{align}\label{lem: hoelder_regualrity_t eq_1}
	&\mathbb{E} 	\left \lvert \langle f(t)-f(s),\varphi \rangle\right \lvert ^{\lambda}\notag\\
	&\leq C(\varphi,\sigma) |t-s|^{\lambda(\frac{1}{2}-\frac{1}{r})} \cdot \left(\left(1+|t-s|^{\frac{\lambda}{2}}\right)\norm{f}^{\lambda}_{L_t^rL_x^pL_v^q}+|t-s|^{\lambda(\frac{1}{2}-\frac{1}{r})} \norm{f}^{2\lambda} _{L_t^rL_x^pL_v^q}\right). 
\end{align}
Therefore, we obtain
\begin{align}\label{lem: hoelder_regualrity_t eq_2}
	\mathbb{E} 	\left[ \langle f,\varphi \rangle\right] ^{\lambda}_{\dot{W}_t^{\kappa,\lambda}} &= \int_0^{T} \int_0^{T}\frac{\mathbb{E} 	\left \lvert \langle f(t)-f(s),\varphi \rangle\right \lvert ^{\lambda}}{|t-s|^{\kappa \lambda + 1}} \d s \d t\notag\\
	& \leq C(\varphi, \sigma, T)\left(\mathbb{E}\norm{f}^{\lambda} _{L_t^rL_x^pL_v^q}+\mathbb{E}\norm{f}^{2\lambda}_{L_t^rL_x^pL_v^q}\right).
\end{align}
	\end{lem}

\begin{proof}
Consider $g$ defined by $g\defeq \int_V K(S)(f)'\d v' - \int_V (K)^{\ast}(S)\d v' f$. We rewrite
	\begin{align}	\begin{split}\label{eq:thm_main:regularity_t}
			&\left \lvert \langle f(t)-f(s),\varphi \rangle\right \lvert \\&= \int_s^t\int_{\R^d}\int_{\R^d} (v\nabla_x \varphi + \mathcal{L}_{\sigma} \varphi) f \d v \d x \d u + \int_s^t\int_{\R^d}\int_{\R^d} \varphi g \d v \d x \d u \\&+ \sum_{k \in \N} \int_s^t\int_{\R^d}\int_{\R^d} f\sigma^k \nabla_v \varphi \d v \d x \d \beta^k(u).
		\end{split}
	\end{align}
	By Hölder's inequality with $\frac{1}{\tilde{r}}=1-\frac{2}{r}$ and a similar calculation to Lemma \ref{lem:a-priori} we estimate the first and second term in \eqref{eq:thm_main:regularity_t} to get $\mathbb{P}$-a.s. 
	\begin{align*}
		\left \lvert\int_s^t\int_{\R^d}\int_{\R^d} (v\nabla_x \varphi + \mathcal{L}_{\sigma} \varphi) f \d v \d x \d u\right \rvert^{\lambda} &\leq C({\varphi,\sigma})\left \lvert t-s \right \rvert ^{\lambda(1-\frac{1}{r})}\norm{f}^{\lambda}_{L_t^rL_x^pL_v^q},\\
		\left \lvert \int_s^t\int_{\R^d}\int_{\R^d} \varphi g \d v \d x \d u \right \rvert ^{\lambda} &\leq C({\varphi,\sigma}) |t-s|^{{\lambda}(1-\frac{2}{r})}\norm{g}^{\lambda} _{L_t^{\tilde{r}'}L_x^{\tilde{p}'}L_v^{\tilde{q}'}}\\
		&\leq C({\varphi,\sigma}) |t-s|^{{\lambda}(1-\frac{2}{r})}\norm{f}^{2\lambda} _{L_t^rL_x^pL_v^q}.
	\end{align*}
	Using the Burkholder-Davis-Gundy inequality for the last term in \eqref{eq:thm_main:regularity_t} we obtain that
	\begin{align*}
		\mathbb{E}	\left \lvert \sum_{k \in \N} \int_s^t\int_{\R^d}\int_{\R^d} f\sigma^k \nabla_v \varphi \d v \d x \d \beta^k(u) \right \rvert ^{\lambda}& \leq \mathbb{E}	\left(\int_s^t\sum_{k \in \N} \left(\int_{\R^d}\int_{\R^d} f\sigma^k \nabla_v \varphi \d v \d x\right)^2 \d u \right) ^{\frac{\lambda}{2}} \\
		&\leq C({\varphi,\sigma}) |t-s|^{\frac{\lambda}{2}(1-\frac{2}{r})} \norm{f}^{\lambda} _{L_t^rL_x^pL_v^q}.
	\end{align*}		
	Thus, we conclude that 
	\begin{align*}
		&\mathbb{E} 	\left \lvert \langle f(t)-f(s),\varphi \rangle\right \lvert ^{\lambda}\\
		&\leq C(\varphi,\sigma) |t-s|^{\lambda(\frac{1}{2}-\frac{1}{r})} \cdot \left(\left(1+|t-s|^{\frac{\lambda}{2}}\right)\norm{f}^{\lambda}_{L_t^rL_x^pL_v^q}+|t-s|^{\lambda(\frac{1}{2}-\frac{1}{r})} \norm{f}^{2\lambda} _{L_t^rL_x^pL_v^q}\right). 
	\end{align*}
	Therefore, for $\lambda >\frac{4r}{r-2}$ and 
	$\kappa \lambda + 1 = \lambda\left(\frac{1}{2}-\frac{1}{r}\right)$, we have \eqref{lem: hoelder_regualrity_t eq_2}.
\end{proof}

\subsection{Proof of Theorem \ref{thm:main}}
In this section, by finding solutions to regularized systems, using the above a-priori-estimates and weak compactness we prove the desired existence result. We will prove both parts of Theorem \ref{thm:main} simultaneously since they only differ in the considered time interval. 

\begin{proof}[Proof of Theorem \ref{thm:main}]
	\textbf{Step 1:}\textit{ Approximating solution.} Given a stochastic chemotactic equation with turning kernel $K$ that satisfies Assumption \ref{ass:turning_kernel} and a nonnegative initial value $f_0 \in L^1(\R^{2d}) \cap L^{a}(\R^{2d})$, we construct sequences $f_0^n$ and $K^n$ that meet the requirements of Assumption \ref{ass: regular} and in addition strongly converge to $f_0$ and $K$ in the appropriate mixed $L^p-$space. 

 We recall that the parameters $a,r,p,q$ are given as stated in Theorem \ref{thm:main}. We further recall that the parameters $\alpha$, $\tilde{q}$, $\tilde{p}$ and $\beta$ are given by $\frac{1}{\alpha}= \frac{1}{d} - \frac{1}{rd}$,  $\frac{1}{\tilde{q}} = \frac{1}{a'} + \frac{1}{d} - \frac{2}{rd}$,  $\frac{1}{\tilde{p}} = \frac{1}{a'} - \frac{1}{d} + \frac{2}{rd}$ and $\frac{1}{\beta} = 1 -  \frac{1}{\alpha}$. The parameters  $\tilde{q}'$ and $\tilde{p}'$ are the dual parameters to  $\tilde{q}'$ and $\tilde{p}'$ and defined by $\frac{1}{\tilde{q}'} = 1- \frac{1}{\tilde{q}}$ and $\tilde{p}'$ analogue.\\
	
	For the sequence $K^n$, we start by truncating $K$ to ensure boundedness and compact support by $\tilde{K}^n(S)(t,x,v,v')\defeq \mathbf{1}_{\{\vert x\vert \leq n,\vert t\vert \leq n,\vert v\vert \leq n,\vert v'\vert \leq n\}}\cdot \mathbf{1}_{\{\norm{K(S)}_{L_{t,x,v,v'}^\infty}\leq n\}} \cdot K(S)(t,x,v,v')$. Next, we truncate by  $\hat{K}^n(S)(t,x,v,v')\defeq \mathbf{1}_{\{\norm{S}_{L_t^{\infty}L_x^{1}}\leq n\}}\tilde{K}^n(S)(t,x,v,v')$
	 and then approximate $\hat{K}^n(S)$ which is Lebesgue integrable by convolution with a smoothing kernel. More precisely, let $\eta(x)=\begin{cases}
	 	c \cdot \exp(-\frac{1}{1-|x|^2}), \quad & |x| < 1\\
	 	0, \quad & \text{else}
	 \end{cases}$ and define $K^n(S) =\hat{K}^n(S)\ast \eta^n$, where $\eta^n(t,x,v,v')=n\cdot \eta(n\cdot t)\cdot n^d\cdot \eta(n\cdot x)\cdot n^d\cdot \eta(n\cdot v)\cdot n^d\cdot \eta(n\cdot {v'})$. Thus, $K^n(S)$ is a smooth function  on a compact set supported in the compact set $V^n\subseteq \R^d$  in the velocity variables with size $|V^n| \leq C(|V|)$. 
	 $K^n(S)$ converges to $K(S)$ in $L_t^rL_x^{\alpha}(B)L_v^{\tilde{q}'}L_{v'}^q\cap L_t^rL_x^{\alpha}(B)L_v^{\alpha}L_{v'}^1 $  for all $S$ in $L_t^rL_x^{\alpha}(B)$ and all compact $B \subseteq \R^d$ \cite{S2013}[Theorem 3.14].  \\
	 Furthermore, as demonstrated in the proof of Lemma \ref{lem:a-priori}, the parameter $\alpha$ is at least as large as the parameters $q$, $\tilde{q}'$ and at least $1$, thereby satisfying Assumption \ref{ass:turning_kernel}. Consequently, we can bound both the $L_t^r[L_x^{\alpha}]L_v^{\tilde{q}'}L_{v'}^q$-norm and the $L_t^r[L_x^{\alpha}]L_v^{\alpha}L_{v'}^1 $-norm of $K(S)$ and $K^n(S)$ uniformly in terms of the $L_t^r[L_x^{\alpha}]$-norms of $S$ and $\nabla S$.\\  	
	 To construct an appropriate sequence $f_0^n$, we set $\tilde{f}_0^n = \mathbf{1}_{\{f_0 \leq n\}}f_0$ 
	 and then set $\hat{f_0^n} = \tilde{f}_0^n+\frac{1}{n}\norm{f_0}_{L_{x,v}^a} (2\pi)^{-\frac{d}{2}} \exp(-\vert x \vert^2 - \vert v \vert ^2) \mathbf{1}_{\bar{V^n}}$
	 	 to ensure, that $f_0^n$ is positive on $\bar{V^n}$ and to bound $\norm{f_0^n}_{L_{x,v}^a}$ in terms of $\norm{f_0}_{L_{x,v}^a}$. Finally, we approximate $\hat{f}_0^n$, which is Lebesgue integrable by convolution with a smoothing kernel. More precisely, let $f_0^n = \hat{f_0^n}\ast\eta^n$ with $\eta^n(x,v)=n^d\cdot \eta(n\cdot x)\cdot n^d\cdot \eta(n\cdot v)$. 
 With this choice, we get $\norm{f_0^n}_{L_{x,v}^a} \leq \norm{f_0}_{L_{x,v}^a}(1+\frac{1}{n}) \leq 2\norm{f_0}_{L_{x,v}^a}$ and $\Vert f_0^n\Vert_{L_{x,v}^1} \le \frac{n+1}{n}\Vert f_0\Vert_{L_{x,v}^1}$ and $f_0^n$ converges to $f_0$ in $L_{x,v}^a\cap L_{x,v}^1$.\\
Consequently, applying Lemma \ref{lem: approx_sol_measurable}, we obtain a solution $\tilde{f}^n$ of \eqref{IVP_stoch} with $K$ and $f_0$ replaced by $K^n$ and $f_0^n$, respectively.
	
	\textbf{Step 2:} \textit{Tightness of $\tilde{f}^n$.} By Lemma \ref{lem:a-priori}, given that $\norm{f_0}_{L_{x,v}^a}$ is sufficiently small, there exists a stopping-time $\tilde{\tau}$ with $\tau \leq\tilde{\tau}\leq T$ and $C$ such that for almost all $\omega$ and for all $n $ we get
	\begin{align*}
		\Vert \tilde{f}^n \Vert_{L_t^r([0,\tilde{\tau}], L_x^p L_v^q)}\le C.
	\end{align*}
	If $\tau$ is deterministic, given that $\norm{f_0}_{L_{x,v}^a}$ is sufficiently small, we set $\tilde{\tau}= T$. 
	The space $L_t^r([0,\tilde{\tau}], L_x^p L_v^q)$ is reflexive. Thus, by the theorem of Banach-Alaoglu the set 
	\begin{align*}
		K\defeq \{f \in L_t^r([0,\tilde{\tau}], L_x^p L_v^q): 	\Vert f \Vert_{L_t^r([0,\tilde{\tau}], L_x^p L_v^q)}\le C \}
	\end{align*}
	is compact with respect to the weak topology. 
	Moreover, by Lemma \ref{lem:conservation of mass}, the norm $\norm{\tilde{f}^n(t)}_{L_{x,v}^1}$ is $\mathbb{P}$-almost surely uniformly bounded by $M$ for all $t$ and all $n$. 
	The set 	\begin{align*}
		\tilde{K}\defeq \{f(t) \in ((L_x^1 L_v^1)')': \norm{f(t)}_{((L_x^1 L_v^1)')'} \leq M \}
	\end{align*}
	is compact with respect to the weak* topology in the bidual space of $L_{x,v}^1$. To show relative compactness, by \cite[Theorem 3.9.1]{EK1986} it is enough to show that $\{\int \varphi \tilde{f}^n \d x\d v\}$ is relatively
	compact (as a family of processes with sample paths in $C_{\R}((0,T))$ for
	each $\varphi \in C_c(\R^{2d})$. The choice of such linear functionals is appropriate since we consider the bidual space of $L_{x,v}^1$. Due to \cite[Theorem 3.7.2]{EK1986} it is enough to show that for every $\eta > 0$ there exists $\delta >0$ such that 
	\begin{align*}
\mathbb{P}\left(\sup _ n \inf_{\{t_i\}} \max _i \sup_{s,t \in [t_{i-1},t_i)}\left \vert\int \tilde{f}^n(t) \varphi \d x \d v - \int \tilde{f}^n(s) \varphi \d x \d v\right \vert \geq \eta\right)\leq \eta
	\end{align*}
		where ${t_i}$ ranges over all partitions of the form $0 = t_0 < t_1 < \dots < t_{k-1} < T \leq t_k$ with $\min_{1\leq i\leq k }(t_i-t_{i-1}) > \delta
$ and $k \geq 1$. This is true by Hölder continuity of the paths by Lemma \ref{lem: hoelder_regularity_t}.

	\textbf{Step 3:} \textit{Tightness of $\tilde{S}^n$.} Let $ B \subseteq \R^d$ be an arbitrary compact set. Next, we aim to show that the random sequence $\tilde{S}^{n}\cdot \mathbf{1}_B$ with values in $L_t^r([0,\tilde{\tau}],[L_x^{\alpha}](B))$ where $\alpha$ is specified in the proof of Lemma \ref{lem:a-priori} is tight with respect to the strong topology. 
	Define the set
	\begin{align*}
		E^l &= \left\{S \in L_t^rL_x^{\alpha} : \norm{\nabla_x S }_{L_t^rL_x^{\alpha}} \leq l, \norm{ S }_{L_t^rL_x^{\alpha}} \leq l\right\}.
	\end{align*}
	Let $\{\varphi_j\}\subseteq C_c^{\infty}(\R^d)$ be a dense subset of $L_x^{{\alpha}'}$ and $\{N^j\}_{j \in \N}$ be a positive, real-valued sequence to be chosen later. Moreover, let $\lambda$ and $\kappa$ satisfy $\lambda >\frac{4r}{r-2}$ and 
	$\kappa \lambda + 1 = \lambda\left(\frac{1}{2}-\frac{1}{r}\right)$. Now, consider the Sobolev-space $W_t^{\kappa, \lambda}([0,T],\R) $ and define the set
	\begin{align*}
		F^l & = \bigcap_{j=1}^{\infty}\left\{S \in L_t^r([0,\tilde{\tau}],L_x^{\alpha}) : \left \lVert \langle S,\varphi_j\rangle \right \rVert_{\dot{W}_t^{\kappa, \lambda}} \leq (lN^j)^{\frac{1}{\lambda}}, \norm{ S }_{L_t^rL_x^{\alpha}} \leq l \right\}.
	\end{align*}
	Let $A^l = E^l\cap F^l$. 
	We aim to show that $A^l$ is a relatively compact set in $L_t^r([0,\tilde{\tau}],L_x^{\alpha}(B))$ with
	\begin{align*}
		\lim_{l\rightarrow \infty} \sup_n \tilde{\mathbb{P}}\left(\tilde{S}^n \notin A^l\right) = 0.
	\end{align*} 
	Therefore, we need to show certain uniform bounds. Firstly, using Lemma \ref{lem:a-priori} and $q \geq 1$ we have 
	\begin{align*}
		\norm{\tilde{S}^n}_{L_t^r([0,\tilde{\tau}],L_x^{\alpha})} \leq \norm{G\ast \tilde{\rho}_n}_{L_t^r([0,\tilde{\tau}],L_x^{\alpha})} \leq \norm{G}_{L_x^b}\cdot\norm{\tilde{\rho}_n}_{L_t^r([0,\tilde{\tau}],L_x^p)} \leq C \norm{\tilde{f}^n}_{L_t^r([0,\tilde{\tau}],L_x^pL_v^q)}\leq C 
	\end{align*}
$\tilde{\mathbb{P}}$-a.s. uniformly in $n$. 
Secondly, using the equation for $\tilde{S}^n$ and, following the approach used in the proof of Lemma \ref{lem:a-priori}, we obtain that there exists a constant $C$, independent of $n$ and $\omega$, such that $\mathbb{P}$-a.s. 
	\begin{align}\label{ineq: grad_S}
		\norm{\nabla_x \tilde{S}^n}_{L_t^r([0,\tilde{\tau}],L_x^{\alpha})}	\leq C.
	\end{align} 
Applying the Chebyshev inequality together with the above inequality \eqref{ineq: grad_S}, we obtain that 
\begin{align*}
	\mathbb{P} \left\{\tilde{S}^n \notin E^l\right\} \leq \frac{C}{l}. 
\end{align*}
	Finally, with Lemma \ref{lem: hoelder_regularity_t} we have 
		\begin{align*}
		\tilde{\mathbb{E} }	\left[ \langle \tilde{f}^n,\varphi \rangle\right] ^{\lambda}_{\dot{W}_t^{\kappa,\lambda}} &= \int_0^{\tilde{\tau}} \int_0^{\tilde{\tau}}\frac{\tilde{\mathbb{E} }	\left \lvert \langle \tilde{f}^n(t)-\tilde{f}^n(s),\varphi \rangle\right \lvert ^{\lambda}}{|t-s|^{\kappa \lambda + 1}} \d s \d t\notag\\
		& \leq C(\varphi, \sigma, T)\left(\tilde{\mathbb{E}}\norm{\tilde{f}^n}^{\lambda} _{L_t^r([0,\tilde{\tau}],L_x^pL_v^q)}+\tilde{\mathbb{E}}\norm{\tilde{f}^n}^{2\lambda}_{L_t^r([0,\tilde{\tau}],L_x^pL_v^q)}\right) \leq C(\varphi,\sigma,T,V).
		\end{align*}
Furthermore, we rewrite
\begin{align}\label{eq: convolution S}
	\langle S(t)-S(s),\varphi_j\rangle =\langle G \ast (\rho(t)-\rho(s)),\varphi_j\rangle = \langle \rho(t)-\rho(s),G\ast\varphi_j\rangle =\langle f(t)-f(s),1_V \times G\ast\varphi_j\rangle. 
\end{align}

Choosing $N^j = 2^jC(\varphi_j, \sigma,T,V)$ and using Lemma \ref{lem: hoelder_regularity_t} together with \eqref{eq: convolution S} we obtain that 
	\begin{align*}
	\tilde{	\mathbb{P}} \left\{\tilde{S}^n \notin F^l\right\} \leq \sum_{j=1}^{\infty}\frac{C(\varphi_j)}{lN^j}\leq \frac{1}{l}\sum_{j=1}^{\infty} 2^{-j} = \frac{1}{l}.
	\end{align*}
	Taking $l \rightarrow \infty$ gives 
	\begin{align*}
		\lim_{l\rightarrow \infty} \sup_n \tilde{\mathbb{P}}\left(\tilde{S}^n \notin A^l\right) = \lim_{l\rightarrow \infty} \frac{1}{l} = 0.
	\end{align*}
	It remains to show that $A^l$ is relatively compact in $L_t^r((0,\tilde{\tau}],L_x^\alpha(B))$. First, by \cite[Theorem 7.1 and Proposition 2.2]{DPV2012} the set $E^l_t \defeq \{S(t, \cdot): S \in E^l\}$ is relatively compact in $L_x^{\alpha}(B)$ for all $t$. Secondly, note that $\lambda > \frac{4r}{r-2}$ implies $\kappa > \frac{1}{\lambda}$.
	Thus, due to \cite[Lemma 5]{S1987} the set $F^l$ is equicontinuous in $t$ and therefore, by a variant of Kolmogorov-Fr\'{e}chet theorem for Banach-spaces \cite[Theorem 1]{S1987} the set $A^l$ is relatively compact in $L_t^r((0,\tilde{\tau}],L_x^\alpha(B))$.
	
\textbf{Step 4:} \textit{Jakubowski space} We next validate that $$\mathcal{X}\defeq [L_t^r([0,\tilde{\tau}],L_x^pL_v^q)]_w\cap C_t([L_{x,v}^1]_w)\times L_t^r([0,\tilde{\tau}],L_x^{\alpha}(B))\times C([0,T],\ell^2)\times [0,T]$$ satisfies the topological assumption in \cite[Theorem 2]{J1997}, that is, that there is  a countable family of functions $\{h_i:\mathcal{X}\rightarrow [-1,1]\}_i$ that is continuous with respect to the corresponding topology and separates points of $\mathcal{X}.$ Each Lebesgue-space $L^p$, respectively $\ell^p$ with $1\leq p<\infty$ is separable and thus, contains a countable, dense subset. Denote by  $\{e^1_i\}_i$ this countable dense subset of $L_t^r([0,\tilde{\tau}],L_x^pL_v^q)$, and, by $\{e^2_i\}_i$ this countable dense subset of $L_x^1L_v^1$,  and, by $\{e^3_i\}_i$ this countable dense subset of $L_t^r([0,\tilde{\tau}],L_x^{\alpha}(B))$, and by $\{e^4_i\}_i$ this countable dense subset of $\ell^2$.  With $\{t_k\}_k = \mathbb{Q}\cap [0,T]$, we define $h_{i_1}(f,S, \beta, \tau)\defeq \norm{f-e^1_i}_{L_t^rL_x^pL_v^q}$, and,  $h_{i_{2_k}}(f,S, \beta, \tau)\defeq \norm{f(t_k)-e^2_i}_{L_x^1L_v^1}$, and  $h_{i_3}(f,S, \beta, \tau)\defeq \norm{S-e^3_i}_{L_t^rL_x^{\alpha}}$, and \\$h_{i_{4_k}}(f,S, \beta, \tau)\defeq \norm{\beta(t_k)-e^4_i}_{\ell^2}$, and $h_{1_5}(f,S, \beta, \tau)\defeq \tau$. With this definitions and $C$ the maximal constant of step $2$ and $3$ the countable set $\frac{1}{T\cdot C}\{h_{i_1},h_{i_{2_k}},h_{i_3},h_{i_{4_k}},h_{1_5}\}_{i_k}$ separates points of $\mathcal{X}$.

\textbf{Step 5:} \textit{Weak and strong convergence of the linear terms.} 
	As a result of Steps 2 and 3, $(\tilde{f}^n,\tilde{S}^n,\{\tilde{\beta}^k\}_k , \tilde{\tau})$ induces tight laws on $\mathcal{X}$.  
	Applying the Skorohod embedding theorem \cite[Theorem 2]{J1997} and working on a subsequence if necessary, there are a new probability space $(\Omega,\mathcal{F},\mathbb{P})$, random variables $(f, S,\{\beta^k\}_k,\tau^{\ast})$, with values in $[L_t^r([0,\tau^{\ast}],L_x^pL_v^q)]_w\cap C_t([L_{x,v}^1]_w)\times L_t^r([0,\tau^{\ast}],L_x^{\alpha}(B))\times C([0,T],\ell^2)\times [0,T]$
 and, a sequence of measurable maps $\tilde{T}^n:\Omega \rightarrow \tilde{\Omega}$ 
such that $(f^n \defeq \tilde{f}^n \circ \tilde{T}^n,S^n\defeq \tilde{S}^n \circ \tilde{T}^n, \{\beta^{n,k}\}_k\defeq \{\tilde{\beta}^{k}\circ \tilde{T}^n\}_k,\tau^{\ast,n}\defeq \tilde{\tau}\circ \tilde{T}^n)$ converges to  $(f, S,\{\beta^k\}_k,\tau^{\ast})$ for all $\omega \in \Omega$. For the explicit construction of the  sequence of maps $\tilde{T^n}$ by using the above defined family of maps $\{h_i\}_i$ we refer to the proof of \cite[Theorem 2]{J1997}. Since $\tilde{f}^n$ is a stochastically strong solution, $f^n$ is also a weak martingale solution to \eqref{IVP_stoch} starting from $f_0$ with noise coefficients $\sigma^k$ with respect to the stochastic basis $(\Omega, \mathcal{F},\mathbb{P},(\mathcal{F}_t^n)_t,\{\beta^{k,n}\}_k)$, where $\mathcal{F}_t^n = (\tilde{T}^n)^{-1} \circ \tilde{\mathcal{F}}_t$. Moreover, $\tau^{\ast,n}$ and $\tau^{\ast}$ are $(\mathcal{F}_t)_t$ stopping-times, since $\tilde{\tau}$ is an $(\tilde{\mathcal{F}}_t)_t$ stopping-time. 

\textbf{Step 6:} \textit{Weak convergence of the nonlinear terms.} 
For any $\varphi \in C_c^{\infty}(\R^d \times V)$, and all $t \in [0, \tau^{\ast}]$, and all $\omega \in \Omega$ we show, that   
\begin{align}
&\int_{0}^{t} \int_{\R^d}\int_{V} \int_{V} K^n(S^n) (f^n)' - (K^n)^{\ast}(S^n)f^n\d v'\varphi(x,v)\d v \d x \d s \notag\\
&\rightarrow \int_{0}^{t} \int_{\R^d}\int_{V} \int_{V} K(S) (f)' - (K)^{\ast}(S)f\d v'\varphi(x,v)\d v \d x \d s.\label{eq: weak_convergence}
 \end{align}
Indeed, we estimate 
\begin{align*}
	&\left \vert \int_{0}^{t} \int_{\R^d}\int_{V}\left[ \int_{V} K^n(S^n) (f^n)' - (K^n)^{\ast}(S^n)f^n\d v'- \int_{V} K(S) f' - K^{\ast}(S)f \d v' \right]\varphi(x,v)\d v \d x \d s \right \vert \\
	& \leq \left \vert \int_{0}^{t\ } \int_{\R^d}\int_{V}\left[ \int_{V} K^n(S^n) (f^n)' -K(S) (f^n)' \d v' \right]\varphi(x,v)\d v \d x \d s \right \vert \\
	& + \left \vert \int_{0}^{t } \int_{\R^d}\int_{V}\left[ \int_{V} K(S) (f^n)' -K(S) f' \d v' \right]\varphi(x,v)\d v \d x \d s \right \vert \\
	& + \left \vert \int_{0}^{t} \int_{\R^d}\int_{V}\left[ \int_{V} K^{\ast}(S)f -K^{\ast}(S)f^n \d v' \right]\varphi(x,v)\d v \d x \d s \right \vert \\
	& + \left \vert \int_{0}^{t} \int_{\R^d}\int_{V}\left[ \int_{V} K^{\ast}(S)f^n -(K^n)^{\ast}(S^n)f^n \d v' \right]\varphi(x,v)\d v \d x \d s \right \vert \\
	&= I+II+III+IV.
\end{align*} 
By continuity of $K$ in $L^{\alpha}(B)$ and since $K^n(S) \rightarrow K(S)$ strongly in $L_t^r([0,\tau^{\ast}],L_x^{\alpha}(B)L_v^{\tilde{q}'}L_{v'}^q)\cap L_t^r([0,\tau^{\ast}],L_x^{\alpha}(B)L_v^{\alpha}L_{v'}^1)$ for every $\varepsilon > 0$ and almost all $\omega$, there exists an integer $n_0 \in \N$ such that for all integers $n \geq n_0$ and all $t\leq \tau^{\ast}$ we have 
\begin{align*}
	I &\leq t^{\tilde{r}}\norm{K^n(S^n) -K(S) }_{L_t^r([0,\tau^{\ast} ],L_x^{\alpha}(B)L_v^{\tilde{q}'}L_{v'}^q)}\norm{f^n}_{L_t^r([0,\tau^{\ast}],L_x^{p}L_v^{q})}\norm{\varphi}_{L_x^{\tilde{p}}L_v^{\tilde{q}}}\\
	&\leq C_{\varphi,t} \left(\norm{K^n(S^n) -K(S^n) }_{L_t^r([0,\tau^{\ast}],L_x^{\alpha}(B)L_v^{\tilde{q}'}L_{v'}^q)} + \norm{K(S^n) -K(S) }_{L_t^r([0,\tau^{\ast}],L_x^{\alpha}(B)L_v^{\tilde{q}'}L_{v'}^q)}\right)\\
	& < \frac{\varepsilon}{4}.
\end{align*}
Since $V$ and $[0,t]$ are compact and $f^n$ is weakly convergent in $L_t^r([0,\tau^{\ast}]L_x^pL_v^q$ for the second term we obtain 
\begin{align*}
	II &\leq \frac{\varepsilon}{8}+\left \vert \int_{V}\int_{0}^{t} \int_{\R^d}\int_{V} K(S) (f^n(s,x,v') - f(s,x,v')) \tilde{\varphi}(s,x,v,v')\d v' \d x \d s \d v\right \vert 
	 < \frac{\varepsilon}{4}.
\end{align*}
Analogously to the second term, for the third term we have
\begin{align*}
	III &\leq \frac{\varepsilon}{8}+\left \vert \int_{0}^{t} \int_{\R^d}\int_{V} \int_{V}K^{\ast}(S) \d v' (f^n(s,x,v) - f(s,x,v)) \tilde{\varphi}(s,x,v)\d v \d x \d s \right \vert < \frac{\varepsilon}{4}.
\end{align*}
We treat the fourth term similar to the first term. Thus, we obtain
\begin{align*}
	IV &\leq t^{\tilde{r}}\norm{ K(S) -K^n(S^n) }_{L_t^r([0,\tau^{\ast}],L_x^{\alpha}(B)L_v^{\alpha}L_{v'}^1)}\norm{f^n}_{L_t^r([0,\tau^{\ast}],L_x^{p}L_v^{q})}\norm{\varphi}_{L_x^{\tilde{p}}L_v^{\beta}}\\
	&\leq C_{\varphi,t}\left(\norm{K(S) -K(S^n) }_{L_t^r([0,\tau^{\ast}],L_x^{\alpha}(B)L_v^{\alpha}L_{v'}^1)}+ \norm{K(S^n) -K^n(S^n) }_{L_t^r([0,\tau^{\ast}],L_x^{\alpha}(B)L_v^{\alpha}L_{v'}^1)}\right)\\
	& < \frac{\varepsilon}{4}.
\end{align*}
Combining these estimates, we obtain the weak convergence of \eqref{eq: weak_convergence}. 
Furthermore, by a similar calculation as for the a-priori estimates in Lemma \ref{lem:a-priori} we obtain $\mathbb{P}$-almost surely \begin{align*}
	\left \vert \int_{0}^{t} \int_{\R^d}\int_{V}\left[ \int_{V} K^n(S^n) (f^n)' - (K^n)^{\ast}(S^n)f^n\d v \right]\varphi(x,v)\d v \d x \d s \right \vert \leq C_{\varphi, t}.
	\end{align*}	
	
\textbf{Step 7:} \textit{Martingale representation.} Using the martingale representation Lemma \cite[Lemma B.2]{PS2018}, it remains to show that there exists a stochastic basis $(\Omega, \mathcal{F^{\tau^{\ast}}}, \mathbb{P},(\mathcal{F}_t^{\tau^{\ast}})_{t=0}^{\tau^{\ast}}, (\beta^k)_{k \in \N})$ such that for all test functions $\varphi \in C_c^{\infty}(\R^d \times V)$, the process $(M^{\tau^{\ast}}_t(\varphi))_{t=0}^{\tau^{\ast}}$ defined by $M^{\tau^{\ast}}_t(\varphi) \defeq M_{t \land \tau^{\ast}}(\varphi)$ with
\begin{align*}
M_t(\varphi) =	\int_{\R^{2d}} f_t\varphi \d x\d v - \int_{\R^{2d}}f_0\varphi \d x \d v - \int_{0}^{t}\int_{\R^{2d}}f(v \cdot \nabla_x \varphi + \mathcal{L}_{\sigma} \varphi) + g\varphi \d x \d v \d s,
\end{align*}
is a continuous $(\mathcal{F}_t^{\tau^{\ast}})_t$ martingale with quadratic variation 
\begin{align}\label{eq: quadratic_variation_def}
\langle \langle M^{\tau^{\ast}}(\varphi), M^{\tau^{\ast}}(\varphi) \rangle \rangle_t &=\sum_{k \in \N}\int_0^{t\land \tau^{\ast}}\left(\int_{\R^{2d}}f_s\sigma^k \cdot \nabla_v \varphi \d x \d v\right)^2 \d s,
\end{align}
and cross variation
\begin{align}\label{eq: cross_variation_def}
\langle \langle M^{\tau^{\ast}}(\varphi),\beta^k \rangle \rangle_t &= \int_0^{t\land \tau^{\ast}}\int_{\R^{2d}} f_s\sigma^k \cdot \nabla_v \varphi \d x \d v \d s.
\end{align}
The rest of the proof is inspired by the proof of \cite[Lemma 3.3]{PS2018}. 
To define the filtration, we first introduce the spaces $(E_t^{\tau^{\ast}})_{t=0}^{\tau^{\ast}}$ given by $E_t^{\tau^{\ast}} = E_{t\land \tau^{\ast}}$ with $$E_t = ([L^r([0,t],L_x^pL_v^q)]_{w}\cap C([0,t],[L_{x,v}^1]_w)) \times [L^{\tilde{r}'}([0,t],L_x^{\tilde{p}'}L_v^{\tilde{q}'})]_{w}\times C([0,t],\ell^2)
.$$ 
Let $r_t : E_{\tau^{\ast}} \rightarrow E_t^{\tau^{\ast}}$ be the corresponding restriction operators, which restricts to functions defined on the time-interval $[0,\tau^{\ast} \land t ]$. Next, let $(G_t^{\tau^{\ast}})_{t=0}^{\tau^{\ast}}=(\int_{0}^{t\land \tau^{\ast}}g(s)\d s)_{t=0}^{\tau^{\ast}}$ and $((G_t^{\tau^{\ast}})^n)_{t=0}^{\tau^{\ast}} = (\int_{0}^{t \land \tau^{\ast}}g^n(s)\d s)_{t=0}^{\tau^{\ast}} $ represent the running time integrals of $g$ and $g^n$, respectively.
Define the $E_{\tau^{\ast}}$-valued random variables $X=(f,G,\{\beta^k\}_k
)$ and $X^n=(f^n,G^n,\{\beta^{k,n}\}_k
)$. \\
We will verify that $f$ is a weak martingale solution on $[0,  \tau^{\ast}]$ relative to the filtration $(\mathcal{F}^{\tau^{\ast}}_t)_t$, where $\mathcal{F}^{\tau^{\ast}}_t = \sigma(r_tX)$. 
To verify that $M_t^{\tau^{\ast}}$ is a martingale, it suffices to show that for all $s<t$ and  $\gamma \in C_b(E_s,\R)$ we have
\begin{align}
	\mathbb{E} (\gamma(r_sX)(M_t^{\tau^{\ast}}(\varphi)-M_s^{\tau^{\ast}}(\varphi)))	&= 0 \label{eq: martingale}.
\end{align}
More precisely, $f(t), G(t)^{\tau^{\ast}}$ and each Brownian motion $\beta^k(t)$ are $\mathcal{F}_t^{\tau^{\ast}}$-adapted by the construction of $E_t^{\tau^{\ast}}$ and therefore, $M_t^{\tau^{\ast}}$ is $\mathcal{F}_t^{\tau^{\ast}}$-adapted. To show that $M_t^{\tau^{\ast}}$ is an $\mathcal{F}_t^{\tau^{\ast}}$-martingale it remains to show that $\mathbb{E}(M_t^{\tau^{\ast}}|\mathcal{F}_s^{\tau^{\ast}}) = M_s^{\tau^{\ast}}$. This is equivalent to $\mathbb{E}((M_t^{\tau^{\ast}}-M_s^{\tau^{\ast}})\mathbf{1}_{A}) = 0$ for all $A \in \mathcal{F}_s^{\tau^{\ast}}$. The formulation \eqref{eq: martingale} is equivalent to this by approximating step-functions by functions that are continuous and bounded. 
To demonstrate that the quadratic variation is given by \eqref{eq: quadratic_variation_def} by a similar argument it is enough to show
\begin{align}
\mathbb{E} \left(\gamma(r_sX)\left(M_t^{\tau^{\ast}}(\varphi)^2-M_s^{\tau^{\ast}}(\varphi)^2\right)\right)	&= \sum_k\mathbb{E} \left(\gamma(r_sX)\int_{s \land{\tau^{\ast}}}^{t\land{\tau^{\ast}}}\left(\int_{\R^{d}}\int_{\R^d}f\sigma^k \nabla_v\varphi \d v \d x\right)^2\d r\right). \label{eq: quadrativ_variation}
\end{align}
Finally, the cross variation \eqref{eq: cross_variation_def} can be shown by verifying the equivalent formulation
\begin{align}
\mathbb{E} \left(\gamma(r_sX)\left(M_t^{\tau^{\ast}}(\varphi)\beta^k_t-M_s^{\tau^{\ast}}(\varphi)\beta^k_s\right)\right)	&= \mathbb{E} \left(\gamma(r_sX)\int_{s \land{\tau^{\ast}}}^{t\land{\tau^{\ast}}}\left(\int_{\R^{d}}\int_{\R^d}\sigma^k \nabla_v \varphi f \d v \d x\right)\d r\right). \label{eq: cross_variation}
\end{align}
Define $(\mathcal{F}_t^{n,\tau^{\ast}})_t$ by $\mathcal{F}_t^{n,\tau^{\ast}} =\sigma(r_tX^n)$ and the continuous $(\mathcal{F}_t^{n,\tau^{\ast}})$ martingale by $M_t^{n,\tau^{\ast}}(\varphi) = M_{t\land \tau^{\ast}}^n(\varphi) $ with 
\begin{align}\label{eq: martingale_approx}
M_t^n(\varphi) \defeq	\int_{\R^{2d}} f^n(t)\varphi dxdv - \int_{\R^{2d}}f_0^n\varphi dxdv - \int_{0}^{t}\int_{\R^{2d}}f^n(v \cdot \nabla_x \varphi + \mathcal{L}_{\sigma} \varphi) + g^n\varphi dxdvds,
\end{align}
where $f^n$ is a stochastically strong, analytically weak solution which exists due to Lemma \ref{lem: approx_sol_measurable}. \\
This implies, that $f^n$ is also a weak martingale solution on $[0,\tau^{\ast}]$ with respect to the stochastic basis $(\Omega, \mathcal{F}^{n,\tau^{\ast}}, \mathbb{P},(\mathcal{F}^{n,\tau^{\ast}}_t)_t, (\beta^{k,n})_{k \in \N})$. Furthermore, $M_t^{n,\tau^{\ast}}(\varphi)$ is a martingale and thus satisfies 
\begin{align}
\mathbb{E} (\gamma(r_sX^n)(M_t^{n,\tau^{\ast}}(\varphi)-M_s^{n,\tau^{\ast}}(\varphi)))	&= 0 \label{eq: martingale_app}.
\end{align}
Its quadratic variation fulfills
\begin{align}
&\mathbb{E} \left(\gamma(r_sX^n)\left(M_t^{n,\tau^{\ast}}(\varphi)^2-M_s^{n,\tau^{\ast}}(\varphi)^2\right)\right)	\notag\\
&= \sum_k\mathbb{E} \left(\gamma(r_sX^n)\int_{s \land{\tau^{\ast}}}^{t\land{\tau^{\ast}}}\left(\int_{\R^{d}}\int_{\R^d}f^n\sigma^k \nabla_v\varphi \d v \d x\right)^2\d r \right) \label{eq: quadrativ_variation_app}.
\end{align}
And, its cross variation satisfies
\begin{align}
&\mathbb{E} \left(\gamma(r_sX^n)\left(M_t^{n,\tau^{\ast}}(\varphi)\beta^{k,n}_t-M_s^{n,\tau^{\ast}}(\varphi)\beta^{k,n}_s\right)\right)	\notag\\&= \mathbb{E} \left(\gamma(r_sX^n)\int_{s \land{\tau^{\ast}}}^{t\land{\tau^{\ast}}}\left(\int_{\R^{d}}\int_{\R^d}\sigma^k \nabla_v \varphi f^n \d v \d x\right)\d r\right) \label{eq: cross_variation_app}.
\end{align}
It remains to show that these equations converge for $n \rightarrow \infty$.
First, we show that for each $t\in [0,\tau^{\ast}]$ the sequence $(M_t^{n,\tau^{\ast}}(\varphi))$ converges to $M_t^{\tau^{\ast}}(\varphi)$ in $L^2(\Omega)$. This follows from several facts. 
For the convergence of the first term of the right-hand-side of \eqref{eq: martingale_approx} we use that $f^n \rightarrow f$ in $L^2(\Omega,C_t([L_{x,v}^1]_w))$ by Step 5 and dominated convergence. This implies $\int_{\R^{d}}\int_{\R^d} f^n(t)\varphi dxdv \rightarrow \int_{\R^{d}}\int_{\R^d} f(t)\varphi dxdv$ in $L^2(\Omega)$.
For the convergence of the initial value term in \eqref{eq: martingale_approx} we use that $f^n_0 \rightarrow f_0$ in $L_{x,v}^a$ with $\Vert f_0^n\Vert_{L_{x,v}^a} \le \left(1+\frac{1}{n}\right)\Vert f_0\Vert_{L_{x,v}^a}$. Consequently, $\int_{\R^{d}}\int_{\R^d} f^n_0\varphi dxdv \rightarrow \int_{\R^{d}}\int_{\R^d} f_0\varphi dxdv$ in $L^2(\Omega)$. To prove convergence of the third term in \eqref{eq: martingale_approx}, we handle the terms involving $g^n$ and$f^n$ separately. By Step 5 shown above, and dominated convergence, for the term involving $g^n$ we prove convergence of
\begin{align*}
&\int_{0}^{t} \int_{\R^d}\int_{V}\int_{V} g^n\varphi(x,v)\d v \d x \d s \\
&= \int_{0}^{t} \int_{\R^d}\int_{V}\left[ \int_{V} K^n(S^n) (f^n)' - (K^n)^{\ast}(S^n)f^n\d v' \right]\varphi(x,v)\d v \d x \d s \\
&\rightarrow 	 \int_{0}^{t} \int_{\R^d}\int_{V}\left[ \int_{V} K(S) f' - K^{\ast}(S)f \d v' \right]\varphi(x,v)\d v \d x \d s = \int_{0}^{t} \int_{\R^d}\int_{V}g\varphi(x,v)\d v \d x \d s
\end{align*}
in $L^2(\Omega)$.
For the term involving $f^n$ we make use of the convergence of $f^n \rightarrow f$ in $[L_t^r([0,\tau^{\ast}])L_x^pL_v^q]_{w}$ for all $\omega$, the uniform boundedness shown in Lemma \ref{lem:a-priori}, and dominated convergence, to obtain that
\begin{align*}
\int_{0}^{t}\int_{\R^{d}}\int_{\R^d}f^n(v \cdot \nabla_x \varphi + \mathcal{L}_{\sigma} \varphi)\d v \d x \d s \rightarrow \int_{0}^{t}\int_{\R^{d}}\int_{\R^d}f(v \cdot \nabla_x \varphi + \mathcal{L}_{\sigma} \varphi)\d v \d x \d s
\end{align*}
in $L^2(\Omega)$.
Combining these facts implies that $(M_t^n(\varphi))$ converges to $M_t(\varphi)$ in $L^2(\Omega)$. 
Moreover, for each $t$, $\{\gamma(r_tX^n)\}$ converges to $\{\gamma(r_tX)\}$ with probability one, while remaining bounded in $L^{\infty}(\Omega).$ This implies that equation \eqref{eq: martingale_app} implies equation \eqref{eq: martingale}. Consequently, $M_t$ defines a martingale. \\
It remains to show the convergence of \eqref{eq: quadrativ_variation_app} and \eqref{eq: cross_variation_app} to \eqref{eq: quadrativ_variation} and \eqref{eq: cross_variation}, respectively. 
Using the product limit Lemma, see \cite[Lemma B.1]{PS2018}, we obtain the convergence of the left-hand-sides, namely
\begin{align*}
\lim_{n \rightarrow \infty} 	\mathbb{E} \left(\gamma(r_sX^n)\left(M_t^n(\varphi)^2-M_s^n(\varphi)^2\right)\right) = \mathbb{E} \left(\gamma(r_sX)\left(M_t(\varphi)^2-M_s(\varphi)^2\right)\right),
\end{align*}
and 
\begin{align*}
\lim_{n \rightarrow \infty} 	\mathbb{E} \left(\gamma(r_sX^n)\left(M_t^n(\varphi)\beta^{k,n}_t-M_s^n(\varphi)\beta^{k,n}_s\right)\right) = \mathbb{E} \left(\gamma(r_sX)\left(M_t(\varphi)\beta^k_t-M_s(\varphi)\beta^{k}_s\right)\right). 
\end{align*}
Furthermore, using the weak convergence of $f^n$ in $C_t([L_{x,v}^1]_w)$ for all $\omega$ and the product limit Lemma \cite[Lemma B.1]{PS2018}, we have the convergence of the right-hand-side of \eqref{eq: cross_variation_app} to the right-hand-side of \eqref{eq: cross_variation}, more precisely
\begin{align*}
\lim_{n \rightarrow \infty}& \mathbb{E} \left(\gamma(r_sX^n)\int_s^t\left(\int_{\R^{d}}\int_{\R^d}\sigma^k \nabla_v \varphi f^n \d v \d x\right)\d r\right) \\= & \mathbb{E} \left(\gamma(r_sX)\int_s^t\left(\int_{\R^{d}}\int_{\R^d}\sigma^k \nabla_v \varphi f \d v \d x\right)\d r\right). 
\end{align*}
This implies that \eqref{eq: cross_variation} is satisfied.
It remains to show that the right-hand-side of \eqref{eq: quadrativ_variation_app} converges to the right-hand-side of \eqref{eq: quadrativ_variation}.
For each $k$ the convergence is true due to the weak convergence of $f^n$ in $C_t([L_{x,v}^1]_w)$ for all $\omega$ and the product limit lemma as above. Moreover, using the boundedness of $\sigma^k$ we bound the sequence uniformly by
\begin{align*}
&\sum_k\mathbb{E} \left( \gamma(r_sX^n)\int_s^t\left(\int_{\R^{d}}\int_{\R^d}f^n\sigma^k \nabla_v\varphi \d v \d x\right)^2\d r\right) \\
&\leq \norm{\gamma}_{C_b(E_s,\R)}\norm{\nabla_v \varphi}_{L_{x,v}^{\infty}} T 
\mathbb{E}\norm{f^n}^2_{L_t^{\infty}(L_{x,v}^1)}	\sum_k \norm{\sigma^k}_{L_{x,v}^{\infty}}^2. 
\end{align*}
By dominated convergence we hence obtain that \eqref{eq: quadrativ_variation_app} implies \eqref{eq: quadrativ_variation}. Thus, $f$ is a weak martingale solution with respect to the filtration $(\mathcal{F}_t^{\tau^{\ast}})_t$. 
\end{proof}

\section{Acknowledgements}
Funded by the Deutsche Forschungsgemeinschaft (DFG, German Research Foundation), Project-ID 317210226, SFB 1283.

	\bibliographystyle{alpha}
	\bibliography{Literatur}
	
	\begin{flushleft}
		\small \normalfont
		\textsc{Benjamin Gess\\
			Institut f\"ur Mathematik,
			Technische Universit\"at Berlin\\
			10623 Berlin, Germany\\
			and\\
			Max--Planck--Institute for Mathematics in the Sciences\\
			04103 Leipzig, Germany.}\\
		\texttt{\textbf{benjamin.gess@mis.mpg.de}}
	\end{flushleft}
	
	\begin{flushleft}
		\small \normalfont
		\textsc{Sebastian Herr\\
				Fakult\"at f\"ur Mathematik, Universit\"at Bielefeld\\
			33615 Bielefeld, Germany.}\\
		\texttt{\textbf{herr@math.uni-bielefeld.de}}
	\end{flushleft}
	
	\begin{flushleft}
	\small \normalfont
	\textsc{Anne Niesdroy\\
		Fakult\"at f\"ur Mathematik, Universit\"at Bielefeld\\
		33615 Bielefeld, Germany.}\\
	\texttt{\textbf{aniesdroy@math.uni-bielefeld.de}}
\end{flushleft}
	
\end{document}